\documentclass[a4paper,12pt]{article}
\usepackage{ayv_paperstyle}
\begin{document}
\title{Approximate public-signal correlated equilibria for nonzero-sum differential games\footnote{This work was funded by the Russian Science Foundation (project no.~17-11-01093).}}
\author{Yurii Averboukh\thanks{Krasovskii Institute of Mathematics and Mechanics,\ \ \texttt{e-mail: ayv@imm.uran.ru, averboukh@gmail.com}}\hspace{6pt}{}\thanks{Ural Federal University}}

\date{}
\maketitle

\begin{abstract}
	We construct an approximate public-signal correlated  equilibrium for a nonzero-sum differential game in the class of stochastic strategies with memory. The construction is based on a solution of an auxiliary nonzero-sum continuous-time stochastic game. 	This class of games includes stochastic differential games and continuous-time Markov games. Moreover, we study the limit of approximate equilibrium outcomes in the case when the auxiliary stochastic games tend to the original deterministic one. We show that it  lies in the convex hull of the set of equilibrium values provided by deterministic  punishment strategies. 
\msccode{91A23, 91A10, 49N70, 91A28}\end{abstract}

\section{Introduction}\label{sect:itroduction}
The paper is concerned with approximate equilibria for two player  differential games. This problem is strongly connected with the theory of  system of Hamilton--Jacobi PDEs. It is proved that if the system of Hamilton--Jacobi PDEs admits a classical solution, then this solution is a Nash value for the corresponding nonzero-sum differential game \cite{Friedman}. Moreover, in this case   one can construct a feedback Nash equilibrium. This property is preserved in some cases when the system of Hamilton--Jacobi PDEs admits only generalized solution (see \cite{Bressan}, \cite{Cardal_Plascatz}). However, up to now there is no existence theorem for the system of Hamilton--Jacobi PDEs. Moreover, Bressan and Shen showed the ill-posedness of  this system~\cite{Bressan_Shen_IJGT}.

A different way to construct Nash equilibria for the nonzero-sum differential game is based on so called punishment techniques. This approach guarantees the existence of Nash equilibria \cite{Cleimenov}, \cite{Kononenko}, \cite{Tolwinski}. Using punishment technique, one can characterize the set of all Nash equilibrium values in the class of deterministic strategies~\cite{Chistyakov},~\cite{Cleimenov},~\cite{Tolwinski}. Certainly, this set comprises the values corresponding to solutions of Hamilton--Jacobi PDEs. However, within the punishment approach one can construct equilibria  those are realized only by incredible threats.  A natural way to select a proper Nash equilibrium is to restrict the attention to the  so called Nash--Pareto solution of nonzero-sum games \cite{Cleimenov}. Unfortunately, this solution concept does not satisfy time consistency principle.   

As it was mentioned in \cite{Bressan_Shen_IJGT}, there are two possibilities to overcome these difficulties. The first way is to introduce some noise i.e. replace the original deterministic system with the stochastic system. The second way is to introduce a cooperation.
In the paper we try to follow both ways. We assume that there exist an auxiliary continuous-time stochastic game with the dynamics close to the original deterministic one and a pair of functions satisfying a stability condition for the auxiliary  game. We use this pair of functions and a slight cooperation to construct an approximate equilibrium in the original differential game. Note that the stability condition is always satisfied if the pair of functions solves the system of Hamilton--Jacobi PDEs for the auxiliary stochastic game. 

We allow auxiliary continuous-time stochastic games with  dynamics   given by  generators of the L\'{e}vy-Khintchine type (see \cite{Kolokoltsov} for the general theory of generators of the L\'{e}vy-Khintchine type).  This class of games includes stochastic differential games and continuous-time Markov games (i.e. games with the dynamics given by a continuous-time Markov chain). Note that both aforementioned cases are well studied. 
The stochastic differential games were studied using a system of parabolic PDEs in \cite{Bensoussan_Frehse_1984}, \cite{Bensoussan_Frehse_2000}, \cite{Friedman_stochastic}, \cite{Friedman_stoch_book}, \cite{Hamadene_Manucci}, \cite{Manucci_1}, \cite{Manucci_2}. Another  approach based on forward-backward stochastic differential equations   was developed for the stochastic differential games in~\cite{BSDE1},~\cite{Hamadene_Mu},~\cite{BSDE2}. Punishment strategies were studied for this type of game in \cite{Cardal_stoch}. Note that approaches based on punishment and forward-backward stochastic differential equations are equivalent  (see \cite{Rainer_2007}).
The nonzero-sum Markov games were  studied in~\cite{Levy}. 

We assume that the players observe   the state of the auxiliary stochastic system  and can use a memory. This leads to  public-signal correlated equilibria in the class of strategies with memory. The auxiliary stochastic game plays the role of a model of the original nonzero-sum differential game. The control with model strategies were first proposed for zero-sum differential games in \cite{NN_PDG}. They were applied to construct approximate equilibria in nonzero-sum differential games (see~\cite{averboukh_JCDS}). In those papers only deterministic models were allowed. The stochastic models for zero-sum differential games were developed in \cite{averboukh_SIAM}, \cite{krasovskii_kotelnokova_unification}, \cite{krasovskii_kotelnokova_guide}, \cite{krasovskii_kotelnokova_intercept}. In the paper we extend the mentioned results to the case of nonzero-sum differential games. 

Note that, although the construction of  public-signal strategies proposed in the paper goes back to the ideas of punishment strategies, it  allows to design approximate equilibria based on the solution of the system of PDEs. 
We especially put an attention to the case when the model of the game is determined by a stochastic differential equation. In this case it is shown that one can construct an approximate equailibrium based on a strong solution of the system of parabolic PDEs.

Additionally, we examine  the limit of approximate equilibrium outcomes  when the model stochastic games tend to the original deterministic one. It is shown that any limit equilibrium outcome  (that is a pair of numbers) lies in the convex hull of the set of equilibrium values provided by deterministic  punishment strategies.

The paper is organized as follows. In Section \ref{sect:preliminaries} we introduce the concept of public-signal correlated  approximate equilibria. The next section is concerned with the formulation of the main result. It states that, given an auxiliary nonzero-sum   continuous-time stochastic game and a pair of continuous functions of position satisfying a stability condition for this auxiliary  game, one can construct an approximate equilibrium for the original game. In Section \ref{sect:sys_HJ} we examine the link between the systems of Hamilton--Jacobi PDEs and the proposed stability condition. First, we consider the case of the auxiliary systems of the general form. It is proved that if a pair of functions is a classical solution of the system of Hamilton-Jacobi PDEs, then it satisfies the stability condition. Further, we study the special case when the auxiliary system is determined by a stochastic differential equation. In this case  we show that if the pair of function is a generalized solution of the system of Hamilton-Jacobi PDEs, then  the stability condition is satisfied. Moreover, we present the example demonstrating that the class of functions satisfying the stability condition is not exhausted by the solutions of the system of Hamilton-Jacobi PDEs even in the case of models given by stochastic differential equations. The limit of the approximate equilibrium outcomes is studied in Section~\ref{sect:limit}.   The other sections are concerned with the proof of the main result. In Section~\ref{sect:construction}, given a pair of functions satisfying stability condition for the auxiliary continuous-time stochastic game, we construct a profile of public-signal correlated  strategies. Its properties are examined in Section~\ref{sect:extremal}. Finally, Section \ref{sect:proof_of_the_main_result} is devoted to the proof of the main result. To this end we show  that the  profile of public-signal correlated strategies designed in Section~\ref{sect:construction} is an approximate equilibrium.

\section{Definitions and assumptions}\label{sect:preliminaries}
We study the nonzero-sum differential game with the dynamics given by 
\begin{equation}
\label{eq:sys_original}
\dot{x}=f_1(t,x,u)+f_2(t,x,v),\ \ t\in [0,T], x\in \rd,\ \ u\in U,\ \ v\in V.
\end{equation} Here $u$ (respectively, $v$) denotes the control of the first (respectively, second) player. We assume that the  purpose of the $i$-th player is to maximize the terminal payoff $\gamma_i(x(T))$. Below we assume that $U$ and $V$ are metric compacts. To simplify notation we will also use the following designation:
\begin{equation}\label{intro:f}
f(t,x,u,v)\triangleq f_1(t,x,u)+f_2(t,x,v). 
\end{equation}

To define the notion of approximate public-signal correlated equilibrium let us introduce some auxiliary definitions.

If $\Upsilon$ is a metric space, then denote by $\mathcal{B}(\Upsilon)$ the corresponding Borel $\sigma$-algebra. Set  $\mathbb{F}_{s,r}\triangleq \mathcal{B}(C([s,r];\rd))$. 
Let $\mathcal{D}$ be a linear subspace  of $C^2(\rd)$ containing $C_b^2(\rd)$, linear functions $x\mapsto \langle a,x\rangle$ and  quadratic functions $x\mapsto \|x-a\|^2$. 	

Informally, the concept of public-signal correlated strategies can be described as follows. We assume that both players at each time observe the random signal that is produced by an external device. Below  this information will be a forecasting of a state of a game being a stochastic model of the original game. The players form their control using this shared information and the history of the game.  

This idea can be formalized in the following way.

\begin{definition}\label{def:strategy} A 6-tuple $\mathfrak{w}=(\Omega,\mathcal{F},\{\mathcal{F}\}_{t\in [t_0,T]},u_{x(\cdot)},v_{x(\cdot)}, P_{x(\cdot)})$ is called a profile of public-signal correlated strategies on $[t_0,T]$ if
	\begin{enumerate}[label=(\roman*)]
		\item $(\Omega,\mathcal{F},\{\mathcal{F}_t\}_{t\in [t_0,T]})$ is a measurable space with a filtration;
		\item for each $x(\cdot)\in C([t_0,T];\rd)$, $P_{x(\cdot)}$ is a probability on $\mathcal{F}$;
		\item for each $x(\cdot)\in C([t_0,T];\rd)$, $u_{x(\cdot)}$ (respectively, $v_{x(\cdot)}$) is a $\{\mathcal{F}_t\}_{t\in [t_0,T]}$-progressively measurable process taking values in $U$ (respectively, $V$);
		\item if $x(t)=y(t)$ for all $t\in [t_0,r]$, then 
		\begin{itemize} 
			\item for any $A\in\mathcal{F}_r$, $P_{x(\cdot)}(A)=P_{y(\cdot)}(A)$,
			\item for any $t\in [t_0,r]$,	$u_{x(\cdot)}(t)=u_{y(\cdot)}(t)$, $v_{x(\cdot)}(t)=v_{y(\cdot)}(t)$ $P_{x(\cdot)}$-a.s.
		\end{itemize}
		\item for any $r$, the restrictions of functions $(x(\cdot),t,\omega)\mapsto u_{x(\cdot)}(t,\omega)$, $(x(\cdot),t,\omega)\mapsto v_{x(\cdot)}(t,\omega)$ on $C([t_0,T];\rd)\times [t_0,r]\times \Omega$ are measurable with respect to $\mathbb{F}_{t_0,T}\otimes \mathcal{B}([t_0,r])\otimes \mathcal{F}_{r}$;
		\item for any $A\in \mathcal{F}$, the function $x(\cdot)\mapsto P_{x(\cdot)}(A)$ is measurable with respect to $\mathbb{F}_{t_0,T}$.
	\end{enumerate}
\end{definition}

Let us briefly comment this definition. First, it states that the choice of probability space is the part of the profile of strategies (condition (i)), whereas the probability of the random signal observed by the players  depends on the sample path  of the game (condition (ii)). Further, we assume that the players' controls are functions of the random signal and the sample path of the game (see condition (iii)). Moreover, conditions (iii) and (iv) mean that the random signal and the players' controls depend on the sample path in the nonanticipative way. Conditions (v) and (vi) are technical; they assure the measurable dependence of the players' controls and the probability of the shared signal on the history of the game. 

The definition of the equilibrium involves  unilateral deviations by the players. Usually this means that the player changes only her  control. However, Definition~\ref{def:strategy} states that the probability space is a part of the profile of strategies. Thus, it is natural to allow the deviating player to choose her own probability space which can also comprise the probability space coming from the original profile of strategies. This lead to the following definition.
\begin{definition}\label{def:deviation} Given a profile of public-signal correlated strategies $\mathfrak{w}=(\Omega,\mathcal{F},\{\mathcal{F}_t\}_{t\in [t_0,T]}, P_{x(\cdot)},u_{x(\cdot)},v_{x(\cdot)})$, we say that a profile of strategies $\mathfrak{w}^c=(\Omega^c,\mathcal{F}^c,\{\mathcal{F}^c_t\}_{t\in [t_0,T]}, P^c_{x(\cdot)},u^c_{x(\cdot)},v^c_{x(\cdot)})$ is an unilateral deviation by the  first (respectively, the second) player if there exists a filtered measurable space  $(\Omega',\mathcal{F}',\{\mathcal{F}'\}_{t\in [t_0,T]})$ such that
	\begin{enumerate}[label=(\roman*)]
		\item $\Omega^c=\Omega\times\Omega'$;
		\item $\mathcal{F}^c=\mathcal{F}\otimes\mathcal{F}'$;
		\item $\mathcal{F}^c_t=\mathcal{F}_t\otimes\mathcal{F}'_t$ for $t\in [t_0,T]$;
		\item for any $x(\cdot)\in C([t_0,T];\rd)$ and any $A\in\mathcal{F}$, $P^c_{x(\cdot)}(A\times \Omega')=P_{x(\cdot)}(A)$;
		\item for any $x(\cdot)$, $t\in [t_0,T]$, $\omega\in\Omega$, $\omega'\in\Omega'$, $v_{x(\cdot)}(t,\omega,\omega')=v_{x(\cdot)}(t,\omega)$ (respectively, $u_{x(\cdot)}(t,\omega,\omega')=u_{x(\cdot)}(t,\omega)$).
	\end{enumerate}
\end{definition}

As above let us briefly comment this definition. Conditions (i)--(iii) means that now the signal consists of two parts: the random signal produced in the original strategy $\omega$ and the additional information $\omega'$. Simultaneously, condition (iv) and (v) states that if the player does not deviates, then she does not observe the additional signal $\omega'$.  Notice that  the public-signal profile of strategies is always a deviation from itself. 


Now let us introduce the motion generated by the public-signal correlated profile of strategies.

\begin{definition}\label{def:motion} Let $t_0\in [0,T]$, $x_0\in\rd$, $\mathfrak{w}=(\Omega,\mathcal{F},\{\mathcal{F}_t\}_{t\in [t_0,T]}, P_{x(\cdot)},u_{x(\cdot)},v_{x(\cdot)})$ be a  profile of public-signal correlated strategies on $[t_0,T]$. We say that a pair $(X(\cdot),P)$ is a realization of the motion generated by  $\mathfrak{w}$ and initial position $(t_0,x_0)$ if
	\begin{enumerate}[label=(\roman*)]
		\item $P$ is a probability on $\mathcal{F}$;
		\item $X(\cdot)$ is a $\{\mathcal{F}_t\}_{t\in [t_0,T]}$-adapted process taking values in $\rd$;
		\item $X(t_0)=x_0$ $P$-a.s.;
		\item for $P$-a.e. $\omega\in\Omega$,
		$$\frac{d}{dt}X(t,\omega)=f_1(t,X(t,\omega),u_{X(\cdot,\omega)}(t,\omega))+f_2(t,X(t,\omega),v_{X(\cdot,\omega)}(t,\omega)). $$
		\item $P_{x(\cdot)}=P(\cdot|X(\cdot)=x(\cdot))$ i.e. given $A\in \mathcal{F}$, 
		$$P(A)=\int_{C([t_0,T];\rd)}P_{x(\cdot)}(A)\chi(d(x(\cdot))), $$ where $\chi$ is a probability on $C([t_0,T];\rd)$ defined by the rule: for any $\mathcal{A}\in\mathbb{F}_{t_0,T}$, $\chi(\mathcal{A})\triangleq P\{\omega:X(\cdot,\omega)\in\mathcal{A}\}$. 
	\end{enumerate}
\end{definition}

Below we say that the profile of strategies $\mathfrak{w}$ is a profile of stepwise strategies if there exists a partition $\Delta=\{t_j\}_{j=0}^r$ of the time interval $[t_0,T]$ such that the equalities $x(t_j)=y(t_j)$, $j=0,1,\ldots, r$ imply that $$P_{x(\cdot)}=P_{y(\cdot)},  \ \ u_{x(\cdot)}=u_{y(\cdot)},\ \  v_{x(\cdot)}=v_{y(\cdot)}.$$ Note that if the profile of strategies $\mathfrak{w}$ is stepwise, then, for any $(t_0,x_0)\in [0,T]\times\rd$, there exists at least one realization.

For a given initial position $(t_0,x_0)$ and a profile of public-signal correlated strategies~$\mathfrak{w}$, we can introduce  upper and lower player's outcomes by the following rules:
$$J_i^+(t_0,x_0,\mathfrak{w})\triangleq \sup\{\mathbb{E}\gamma_i(X(T)):(X(\cdot),P)\text{ generated by }\mathfrak{w}\text{ and }(t_0,x_0)\}, $$
$$J_i^-(t_0,x_0,\mathfrak{w})\triangleq \inf\{\mathbb{E}\gamma_i(X(T)):(X(\cdot),P)\text{ generated by }\mathfrak{w}\text{ and }(t_0,x_0)\}. $$
Here $\mathbb{E}$ denotes the expectation according to the probability $P$.

\begin{definition}\label{def:nash} We say that a profile of public-signal correlated strategies $\mathfrak{w}^*$ is a public-signal correlated $\varepsilon$-equilibrium  at the position $(t_0,x_0)\in [0,T]\times\rd$ if, for any profile of strategies $\mathfrak{w}^i$ that is an unilateral deviation from $\mathfrak{w}^*$ by the player $i$, the following inequality holds true:
	$$J^+_i(t_0,x_0,\mathfrak{w}^i)\leq J^-_i(t_0,x_0,\mathfrak{w}^*)+\varepsilon. $$
\end{definition}

To construct a public-signal correlated $\varepsilon$-equilibrium  we use a solution of an auxiliary stochastic game with a dynamics  determined by a generator of the L\'{e}vy-Khintchine type $\Lambda_t[u,v]$. The general theory of such stochastic processes described by generator of the L\'{e}vy-Khintchine type is presented in \cite{Kol_markov}. Assume that, for each $t\in [0,T]$, $u\in U$, $v\in V$, $\Lambda_t[u,v]$ is  an operator from $\mathcal{D}$ to $C(\rd)$ defined by the rule:
\begin{multline}\label{intro:Lambda}(\Lambda_t[u,v]\phi)(x)\triangleq \frac{1}{2}\langle G(t,x,u,v)\nabla,\nabla\rangle\phi(x)+\langle b(t,x,u,v),\nabla\rangle\phi(x)\\+\int_{\rd}[\phi(x+y)-\phi(x)-\langle y,\nabla\phi(x)\rangle\mathbf{1}_{B_1}(y)]\nu(t,x,u,v,dy). \end{multline}
Here $B_1$ stands for the ball of radius $1$ centered at the origin; for each $t\in [0,T]$, $x\in\rd$, $u\in U$, $v\in V$, $G(t,x,u,v)$ is a nonnegative symmetric $d\times d$-matrix, $b(t,x,u,v)$ is a $d$-dimensional vector, $\nu(t,x,u,v,\cdot)$ is a measure on $\rd$ such that $\nu(t,x,u,v,\{0\})=0$.
\begin{remark}
	Let us briefly explain the meaning of the coefficients in (\ref{intro:Lambda}).
	Here $b$ stands for the deterministic evolution, $G$ is a squared diffusion coefficient, whereas $\nu(t,x,u,v,\cdot)$ denotes the intensity of jumps. To illustrate this  let us consider two examples. First, assume that the stochastic process $Y(t)$ is determined by the controlled stochastic differential equation
	$$dY(t)=b(t,Y(t),u(t),v(t))dt+\sigma(t,Y(t),u(t),v(t))dW_t. $$ In this case $G(t,x,u,v)=\sigma(t,x,u,v)\sigma^T(t,x,u,v)$, $\nu(t,x,u,v,dy)\equiv 0$. 
	
	Now, let us consider the pure jump processes. In this case (\ref{intro:Lambda}) takes the form 
	$$(\Lambda_t[u,v]\phi)(x)=\int_{\rd}[\phi(x+y)-\phi(x)]\nu(t,x,u,v,dy). $$ This case can be interpreted as follows. Given the state $x$, the value $\nu(t,x,u,v,\rd)dt+o(dt)$ is a probability of the jump  from $x$ on $[t,t+dt]$, whereas $\nu(t,x,u,v,\Upsilon)/\nu(t,x,u,v,\rd)+o(dt)$ is a probability of transition to $\Upsilon$ if the jump takes place.
\end{remark} 

Furthermore, we assume that the objective function of  the player $i$ in the auxiliary stochastic game   is  equal to
\begin{equation}\label{intro:aux_payoff}
\mathbb{E}\left[\gamma_i(X(T))+\int_{t_0}^Th_i(t,X(t),u(t),v(t))dt\right]. 
\end{equation}

Denote
\begin{equation}\label{intro:Sigma}
\Sigma(t,x,u,v)\triangleq \mathrm{tr}G(t,x,u,v)+\int_{\rd}\|y\|^2\nu(t,x,u,v,dy), 
\end{equation}
\begin{equation}\label{intro:g}
g(t,x,u,v)\triangleq b(t,x,u,v)+\int_{\rd\setminus B_1}y\nu(t,x,u,v,du). 
\end{equation}

We assume that the following conditions hold true:
\begin{enumerate}[label=(L\arabic*)]
	\item $U$, $V$ are metric compacts;
	\item $f_1$, $f_2$ $G$, $b$, $\gamma_1$, $\gamma_2$, $h_1$, $h_2$ are continuous and bounded;
	\item for any $\phi\in \mathcal{D}$, the function $[0,T]\times\rd\times U\times V\ni(t,x,u,v)\mapsto\int_{\rd}\phi(y)\nu(t,x,u,v,dy)$ is continuous. 
	\item there exists a function $\alpha(\cdot):\mathbb{R}\rightarrow [0,+\infty)$ such that $\alpha(\delta)\rightarrow 0$ as $\delta\rightarrow 0$ and, for any $t,s\in [0,T]$, $x\in\rd$, $u\in U$, $v\in V$,
	$$\|f(t,x,u,v)-f(s,x,u,v)\|\leq \alpha(t-s), $$
	$$\|g(t,x,u,v)-g(s,x,u,v)\|\leq\alpha(t-s); $$
	
	\item there exists a constant $M$ such that, for any $t\in [0,T]$, $x\in\rd$, $u\in U$, $v\in V$,
	$$\|f(t,x,u,v)\|\leq M,\ \ \|g(t,x,u,v)\|\leq M; $$
	\item\label{cond:lip} there exists a  constant $K>0$ such that, for any $t\in [0,T]$, $x',x''\in\rd$, $u\in U$, $v\in V$,
	$$\|f(t,x',u,v)-f(t,x'',u,v)\|\leq K\|x'-x''\|,$$
	$$\|g(t,x',u,v)-g(t,x'',u,v)\|\leq K\|x'-x''\|; $$
	\item there exists a constant $R>0$ such that, for any $x',x''\in\rd$, $i=1,2$,
	$$|\gamma_i(x')-\gamma_i(x'')|\leq R\|x'-x''\|; $$
	\item\label{cond:delta} for any $t\in [0,T]$, $x\in\rd$, $u\in U$, $v\in V$, $$|\Sigma(t,x,u,v)|\leq \delta^2,$$
	$$\|f(t,x,u,v)-g(t,x,u,v)\|^2\leq 2 \delta^2,$$
	$$|h_i(t,x,u,v)|\leq \delta. $$	 
\end{enumerate}
In condition \ref{cond:delta} $\delta$ is a small parameter.

\section{Approximate equilibrium and stability condition}\label{sect:main_result}	The main result of the paper involves a pair of functions $(c_1,c_2)$ satisfying a stability property (see Condition ($\mathcal{C}$) below). Roughly speaking one can consider $(c_1,c_2)$ as  value function of the nonzero-sum continuous-time stochastic game with the dynamics determined by the generator $\Lambda_t[u,v]$ and objective functions of the players given by (\ref{intro:aux_payoff}). The link between this stability property and system of Hamilton-Jacobi PDEs is given in Section~\ref{sect:sys_HJ}. To formulate the stability condition we introduce the notion of controlled system admissible for the generator $\Lambda_t[u,v]$. This notion uses the relaxed stochastic control of both players first introduced in \cite{Subb_Chen}. 

Denote the set of probabilities on a Polish space  $\Upsilon$ by $\mathrm{rpm}(\Upsilon)$. We endow $\mathrm{rpm}(\Upsilon)$ with the narrow topology i.e.  $\{\chi_n\}_{n=1}^\infty\subset\mathrm{rpm}(\Upsilon)$ converges to $\chi\in\mathrm{rpm}(\Upsilon)$ iff, for any $\phi\in C_b(\Upsilon)$
$$\int_{\Upsilon}\phi(\upsilon)\chi_n(d\upsilon) \rightarrow \int_{\Upsilon}\phi(\upsilon)\chi(d\upsilon)\text{ as }n\rightarrow\infty.$$
This space is Polish \cite{Billingsley}, \cite{Warga}. Moreover, $\mathrm{rpm}(\Upsilon)$ is  compact when $\Upsilon$ is compact.  The mapping $z\mapsto\delta_z$ provides a natural embedding of $\Upsilon$ into $\mathrm{rpm}(\Upsilon)$.  Hereinafter $\delta_z$ stands for the Dirac measure concentrated at $z$. 

A stochastic process taking values in $\mathrm{rpm}(U)$ (respectively, in $\mathrm{rpm}(V)$) is a relaxed stochastic control of the first (respectively, second) player.  Furthermore,  the stochastic process taking values in  $\mathrm{rpm}(U\times V)$ is a  relaxed stochastic controls of both players. If $\mu(t)$ (respectively, $\nu(t)$) is a relaxed control of the first (respectively, second player), we write $\mu(t,du)$ (respectively, $\nu(t,dv)$) instead of $\mu(t,\omega)(du)$ (respectively, $\nu(t,\omega)(dv)$). Analogously, if $\eta(t)$ is a relaxed control of both players, we write $\eta(t,d(u,v))$ for $\eta(t,\omega)(d(u,v))$.

Now let us introduce the notion of controlled system going back to  \cite{fleming_soner}, \cite{GihmanSkorokhod}. This notion generalize the standard notion of deterministic controlled system. The main difference is that in the stochastic case we assume that the control and the motion of the system are stochastic processes defined on some filtered probability space, when the dynamics is determined by solution of martingale problem (see condition \ref{item_def:martingale} of Definition \ref{def:control_system} below). 

\begin{definition}\label{def:control_system} Let $s,r\in [0,T]$, $s<r$. We say that a 6-tuple $(\Omega,\mathcal{F},\{\mathcal{F}\}_{t\in [s,r]},P,\eta,X)$ is a controlled system on $[s,r]$ admissible for the generator $\Lambda_t[u,v]$ if
	\begin{enumerate}[label=(\roman*)]
		\item $(\Omega,\mathcal{F},\{\mathcal{F}_t\}_{t\in [s,r]},P)$ is a filtered probability space;
		\item $\eta$ is a $\{\mathcal{F}_t\}_{t\in [s,r]}$-progressively measurable stochastic process taking values in $\mathrm{rpm}(U\times V)$;
		\item $X$ is a  $\{\mathcal{F}_t\}_{t\in [s,r]}$-adapted stochastic process taking values in $\rd$;
		\item\label{item_def:martingale} for any $\phi\in \mathcal{D}$, the process
		$$\phi(X(t))-\int_s^t\int_{U\times V}(\Lambda_\tau[u,v]\phi)(X(\tau))\eta(\tau,d(u,v))d\tau $$ is a $\{\mathcal{F}_t\}_{t\in [s,r]}$-martingale.
	\end{enumerate}
\end{definition}

The following stability condition plays a key role in the construction of the approximate public-signal correlated equilibrium. 

\begin{definition}\label{def:cond_C} Let $c_1,c_2:[0,T]\times\rd\rightarrow\mathbb{R}$ be continuous functions. We say that the pair $(c_1,c_2)$ satisfies \textit{Condition $(\mathcal{C})$} if, for any $s,r\in [0,T]$, $s<r$, there exists a filtered measurable space $(\widehat{\Omega}^{s,r},\widehat{\mathcal{F}}^{s,r},\{\widehat{\mathcal{F}}^{s,r}_t\}_{t\in [s,r]})$ satisfying the following properties:
	\begin{list}{(\roman{tmp})}{\usecounter{tmp}}
		\item\label{cond_C:both} given $y\in\rd$, one can find processes $\eta_y^{s,r}$, $\widehat{Y}_y^{s,r}$ and a probability $\widehat{P}_y^{s,r}$ such that the 6-tuple
		$(\widehat{\Omega}^{s,r},\widehat{\mathcal{F}}^{s,r},\{\widehat{\mathcal{F}}^{s,r}_t\}_{t\in [s,r]},\widehat{P}_y^{s,r},\eta_y^{s,r}, \widehat{Y}_y^{s,r})$ is a control system admissible for $\Lambda_t[u,v]$ and, for $i=1,2$,
		$$\widehat{\mathbb{E}}_{y}^{s,r}\left[c_i(r,\widehat{Y}_{y}^{s,r}(r))+ \int_s^r\int_{U\times V}h_i(t,\widehat{Y}_y^{s,r}(t),u,v)\eta_y^{s,r}(t,d(u,v))dt\right]=c_i(s,y); $$
		
		\item\label{cond_C:second_dev} for any $y\in\rd$ and $v\in V$, one can find a relaxed stochastic control of the first player $\mu_{y,v}^{s,r}$, a process $\overline{Y}_{y,v}^{1,s,r}$ taking values in    $\rd$ and a probability $\overline{P}_{y,v}^{1,s,r}$ such that the 6-tuple
		$(\widehat{\Omega}^{s,r},\widehat{\mathcal{F}}^{s,r},\{\widehat{\mathcal{F}}^{s,r}_t\}_{t\in [s,r]},\overline{P}_{y,v}^{1,s,r},\mu_{y,v}^{s,r}\otimes \delta_v, \overline{Y}_{y,v}^{1,s,r})$ is a control system admissible for $\Lambda_t[u,v]$ and 
		$$ \overline{\mathbb{E}}_{y,v}^{1,s,r}\left[c_2(r,\overline{Y}_{y,v}^{1,s,r}(r))+ \int_s^r\int_{ U}h_2(t,\overline{Y}_{y,v}^{1,s,r}(t),u,v)\mu_{y,v}^{s,r}(t,du)dt\right]\leq c_2(s,y); $$
		
		\item\label{cond_C:firts_dev} given $y\in\rd$ and $u\in U$, one can find a second player's relaxed stochastic control $\nu_{y,u}^{s,r}$, a process $\overline{Y}_{y,u}^{2,s,r}$ and a probability $\overline{P}_{y,u}^{2,s,r}$ such that the 6-tuple
		$(\widehat{\Omega}^{s,r},\widehat{\mathcal{F}}^{s,r},\{\widehat{\mathcal{F}}^{s,r}_t\}_{t\in [s,r]},\overline{P}_{y,u}^{2,s,r},\delta_u\otimes\nu_{y,u}^{s,r}, \overline{Y}_{y,u}^{2,s,r})$ is a control system admissible for $\Lambda_t[u,v]$ and 
		$$\overline{\mathbb{E}}_{y,u}^{2,s,r}\left[c_1(r,\overline{Y}_{y,u}^{2,s,r}(r))+ 
		\int_s^r\int_{ V}h_1(t,\overline{Y}_{y,u}^{2,s,r}(t),u,v)\nu_{y,u}^{s,r}(t,dv)dt\right]\leq c_1(s,y). $$
		
	\end{list}
	Here $\widehat{\mathbb{E}}_{y}^{s,r}$ (respectively, $\overline{\mathbb{E}}_{y,u}^{1,s,r}$, $\overline{\mathbb{E}}_{y,u}^{2,s,r}$) denotes the expectation according to the probability $\widehat{P}_{y}^{s,r}$ (respectively, $\overline{P}_{y,u}^{1,s,r}$, $\overline{P}_{y,u}^{2,s,r}$).
\end{definition}

Informally speaking, the meaning of 	Condition $(\mathcal{C})$ is as follows. 
The first part of this condition means that both players can maintain the  value $(c_1(s,y),c_2(s,y))$ on the time interval $[s,r]$ choosing an appropriate controlled stochastic system. Parts~(ii),~(iii) mean that if  player $i$ picks a constant control on $[s,r]$, then the other player can find a controlled system such that the outcome of the player $i$ on $[s,r]$ is not greater than $c_i(s,y)$. Here we assume that the terminal part of the $i$-th player's reward on $[s,r]$ is given by $c_i(r,\cdot)$. Additionally, to avoid  technical issues we assume that all mentioned controlled systems exploit the same filtered measurable space. Notice that Condition~$(\mathcal{C})$ is an extension of the notion of $u$-stability first proposed in~\cite{NN_PDG} to examine zero-sum differential games. The variant of Condition~$(\mathcal{C})$ for the case  when $\Lambda_t[u,v]$ is given only by deterministic evolution was considered in~\cite{averboukh_JCDS}.

Further, let \begin{equation}\label{intro:beta}
\beta\triangleq (5+2K),
\end{equation}
\begin{equation}\label{intro:C}
C\triangleq 2\sqrt{ Te^{\beta T}}. 
\end{equation}

\begin{theorem}\label{th:near_Nash}
	Let continuous functions $c_1,c_2:[0,T]\times\rd\rightarrow\mathbb{R}$ be such that
	\begin{itemize}
		\item $c_i(T,x)=\gamma_i(x)$;
		\item $(c_1,c_2)$ satisfies Condition $(\mathcal{C})$.
	\end{itemize} Then, for any $(t_0,x_0)\in [0,T]\times\rd$, and $\varepsilon> (RC+T)\delta$, there exists a   profile of public-signal correlates strategies $\mathfrak{w}^*$ that is $\varepsilon$-equilibrium  at $(t_0,x_0)$. Moreover, if $X^*$ and $P^*$ are generated by $\mathfrak{w}^*$ and $(t_0,x_0)$, $\mathbb{E}^*$ denotes the expectation according to $P^*$, then
	$$|\mathbb{E}^*\gamma_i(X^*(T))-c_i(t_0,x_0)|\leq\varepsilon. $$
\end{theorem}
To prove this Theorem we introduce the profile of stepwise public-signal correlated strategies in Section \ref{sect:construction}. The construction involves an auxiliary stochastic processes   those can be regarded as models of the original game. The properties of these models are examined in Section \ref{sect:extremal}. They play the crucial role in the proof of Theorem \ref{th:near_Nash} that is presented in Section \ref{sect:proof_of_the_main_result}.

\section{System of Hamilton-Jacobi PDEs}\label{sect:sys_HJ}

In this section we specify Condition~$(\mathcal{C})$ and provide a link between this condition and  a system of Hamilton-Jacobi PDEs. First, we consider the case when the dynamics of the auxiliary system is given by the generator of the general form. For this case we prove that if $c_1,c_2$ is a classical solution to the system of Hamilton-Jacobi equations, then it satisfies Condition~$(\mathcal{C})$. Additionally, we consider the specific case when the dynamics of  the auxiliary system is determined by the stochastic differential equation. In this case we prove that if $(c_1,c_2)$ is the generalized (strong) solution of the system of Hamilton-Jacobi equations, then Condition  $(\mathcal{C})$ holds.

\begin{theorem}\label{th:smooth}
	Assume that the functions $c_1,c_2:[0,T]\times\rd\rightarrow\mathbb{R}$, $u^0:[0,T]\times\rd\rightarrow U$, $v^0:[0,T]\times \rd\rightarrow V$ satisfy the following conditions:
	\begin{enumerate}
		\item the functions $c_1$, $c_2$ are of the class $C^2$;
		\item   for $i=1,2$,
		\begin{equation}\label{eq:system_HJB}
		\begin{split}
		&\frac{\partial c_i}{\partial t}+\Lambda_t[u^0(t,x),v^0(t,x)]c_i(t,x)+h_i(t,x,u^0(t,x),v^0(t,x))=0,\\ &c_i(T,x)=\gamma_i(x).
		\end{split}
		\end{equation}
		\item \begin{multline}\label{equal:u_o}\Lambda_t[u^0(t,x),v^0(t,x)]c_i(t,x)+h_i(t,x,u^0(t,x),v^0(t,x))\\=\max_{u\in U}\left[\Lambda_t[u,v^0(t,x)]c_i(t,x)+h_i(t,x,u,v^0(t,x))\right], \end{multline}
		\begin{multline}\label{equal:v_o}\Lambda_t[u^0(t,x),v^0(t,x)]c_i(t,x)+h_i(t,x,u^0(t,x),v^0(t,x))\\=\max_{v\in V}\left[\Lambda_t[u^0(t,x),v]c_i(t,x)+h_i(t,x,u^0(t,x),v)\right]. \end{multline}
		\item given $[s,r]\in [0,T]$, $s<r$, there exist solutions of the martingale problems on $[s,r]$ for the generators $\Lambda_t[u^0(t,\cdot),v^0(t,\cdot)]$, $\Lambda_t[u,v^0(t,\cdot)]$, $\Lambda_t[u^0(t,\cdot),v]$, $u\in U$, $v\in V$; moreover, one can find a common filtered measurable space with a filtration suitable for all mentioned problems.
	\end{enumerate}
	Then, the pair $(c_1,c_2)$ satisfies Condition $(\mathcal{C})$. In particular, for any $(t_0,x_0)\in [0,T]\times\rd$, $\varepsilon> (RC+T)\delta$, there exists the public-signal correlated $\varepsilon$-equilibrium  at $(t_0,x_0)$.
\end{theorem} 
\begin{remark}\label{remark:HJ}
	Notice that (\ref{eq:system_HJB}) is a system of Hamilton-Jacobi PDEs corresponding to the generator $\Lambda_t[u,v]$ and objective functions (\ref{intro:aux_payoff}), whereas equations (\ref{equal:u_o}), (\ref{equal:v_o}) mean that $u^0(t,x)$ and $v^0(t,v)$ are Nash equilibrium feedback strategies.
\end{remark}

\begin{proof}[Proof of Theorem \ref{th:smooth}]
	Let $(\widetilde{\Omega}^{s,r},\widetilde{\mathcal{F}}^{s,r}, \{\widetilde{\mathcal{F}}_t^{s,r}\}_{t\in [s,r]})$ be a filtered measurable space suitable for the solution of the martingale problems for the generators $\Lambda_t[u^0(t,\cdot),v^0(t,\cdot)]$, $\Lambda_t[u,v^0(t,\cdot)]$, $\Lambda_t[u^0(t,\cdot),v]$. Put $\widehat{\Omega}^{s,r}\triangleq \widetilde{\Omega}^{s,r}$, $\widehat{\mathcal{F}}^{s,r}\triangleq \widetilde{\mathcal{F}}^{s,r}$, $\widehat{\mathcal{F}}^{s,r}_t\triangleq \widetilde{\mathcal{F}}^{s,r}_t$.  Furthermore, given $y\in\rd$, let $\widetilde{P}^{0,s,r}_y$, $\widetilde{Y}^{0,s,r}_y$ be  such that $(\widetilde{\Omega}^{s,r},\widetilde{\mathcal{F}}^{s,r}, \{\widetilde{\mathcal{F}}_t^{s,r}\}_{t\in [s,r]},\widetilde{P}^{0,s,r}_y, \widetilde{Y}^{0,s,r}_y)$ solves the martingale problem for the generator $\Lambda_t[u^0(t,\cdot),v^0(t,\cdot)]$ and initial position $(s,y)$.  Set $\widehat{P}^{s,r}_y\triangleq \widetilde{P}^{0,s,r}_y$, $\widehat{Y}^{s,r}_y\triangleq \widetilde{Y}^{0,s,r}_y$, $\eta^{s,r}_y(t)\triangleq \delta_{u^0(t,\widetilde{Y}^{0,s,r}_y(t))}\times\delta_{v^0(t,\widetilde{Y}^{0,s,r}_y(t))}$. Using the standard dynamical programming arguments, one can prove that the part (i) of Condition~$(\mathcal{C})$ is fulfilled. 
	
	To prove the second part of Condition ($\mathcal{C}$) use the dynamic programming principle and (\ref{eq:system_HJB}), (\ref{equal:u_o}) letting $\overline{P}^{1,s,r}_{y,v}\triangleq \widetilde{P}^{1,s,r}_{y,v}$, $\overline{Y}^{1,s,r}_{y,v}\triangleq \widetilde{Y}^{1,s,r}_{y,v}$, $\mu^{s,r}_{y,v}(t)\triangleq \delta_{u^0(t,\overline{Y}^{1,s,r}_{y,v}(t))}$. Here $\widetilde{P}^{1,s,r}_{y,v}$ and $\widetilde{Y}^{1,s,r}_{y,v}$ are, respectively, a probability on $\widetilde{\mathcal{F}}^{s,r}$ and  a stochastic process defined on $\widetilde{\Omega}^{s,r}$ such that $(\widetilde{\Omega}^{s,r},\widetilde{\mathcal{F}}^{s,r}, \{\widetilde{\mathcal{F}}_t^{s,r}\}_{t\in [s,r]},\widetilde{P}^{1,s,r}_{y,v}, \widetilde{Y}^{1,s,r}_{y,v})$ solves the martingale problem for the generator $\Lambda_t[u^0(t,\cdot),v]$ and initial position $(s,y)$.
	
	The third part is proved in the same way. 
\end{proof}

Now let us restrict the attention to the case when the auxiliary system is given by the controlled stochastic differential equation
\begin{equation}\label{sys:2nd_order_full}
dX(t)=b(t,X(t),u(t),v(t))dt+\sigma(t,X(t),u(t),v(t))dW(t)
\end{equation}  where $W(t)$ stands for the $m$-dimensional Wiener process and $\sigma:[0,T]\times\rd\times U\times V\rightarrow \mathbb{R}^{d\times m}$. This corresponds to the generator 
$$\Lambda_t[u,v]\phi(x)=\frac{1}{2}\langle G(t,x,u,v)\nabla,\nabla\rangle\phi(x)+\langle b(t,x,u,v),\nabla\rangle\phi(x),  $$ with $G(t,x,u,v)=\sigma(t,x,u,v)\sigma^T(t,x,u,v)$. Thus, system (\ref{eq:system_HJB}) is now the system of second order PDEs:
\begin{equation}\label{eq:system_HJB2}
\begin{split}
\frac{\partial c_i}{\partial t}+\frac{1}{2}\langle G(t,x&,u^0(t,x),v^0(t,x))\nabla,\nabla\rangle c_i(t,x)\\+\langle b(t,&x,u^0(t,x),v^0(t,x)),\nabla\rangle c_i(t,x)+h_i(t,x,u^0(t,x),v^0(t,x))=0,\\ c_i(T,x)=\gamma_i(x&).
\end{split}
\end{equation} 

However, to apply Theorem \ref{th:smooth} directly to system (\ref{sys:2nd_order_full}) we are to assume that system (\ref{eq:system_HJB2}) admits a classical solution. That is rather restrictive assumption. For the case when $\sigma$ does not depend on $u$ and $v$ this assumption can be weakened. This result is proved by the dynamic programming arguments. It is a counterpart of \cite[Theorem 17.2.1]{Friedman_stoch_book}, \cite[Theorem 4.1]{Hamadene_Manucci}, \cite[Theorem 4.1]{Manucci_1}, \cite[Theorem 3.6]{Manucci_2}.

We assume that  the dynamics of the auxiliary system is determined by the stochastic differential equation of the following form:
\begin{equation}\label{eq:stoch_sys}
dX(t)=b(t,X(t),u(t),v(t))dt+\sigma(t,X(t))dW(t).
\end{equation} 
This system corresponds to the generator
\begin{equation}\label{intro:generator_2nd_specific}
\Lambda_t[u,v]\phi(x)=\langle b(t,x,u,v),\nabla \phi(x)\rangle+\frac{1}{2}\langle G(t,x)\nabla,\nabla\phi(x)\rangle,
\end{equation} where $G(t,x)=\sigma(t,x)\sigma^T(t,x)$.

We will study the link between the generalized solution of the system of parabolic equation corresponding to the game with  dynamics (\ref{intro:generator_2nd_specific}) and condition $(\mathcal{C})$ under the following additional assumptions:
\begin{enumerate}[label=(A\arabic*)]
	\item $W(t)$ is a $d$-dimensional Wiener process, $\sigma$ is $d\times d$-matrix;
	\item there exists a constant $C_0>0$ such that, for any $t\in [0,T]$, $x\in\rd$, $C_0^{-1}I\leq G(t,x)\leq C_0I$;
	\item $\sigma$ is Lipschitz continuous w.r.t. $x$;
	\item there exist measurable functions $u^N(t,x,p_1,p_2)$, $v^N(t,x,p_1,p_2)$  taking values in $U$ and $V$ respectively such that, for any $t\in [0,T]$, $x,p_1,p_2\in \rd$, $u\in U$, $v\in V$.
	$$\mathcal{H}_1(t,x,p_1,u^N(t,x,p_1,p_2),v^N(t,x,p_1,p_2))\geq \mathcal{H}_1(t,x,p_1,u,v^N(t,x,p_1,p_2)), $$
	$$\mathcal{H}_2(t,x,p_2,u^N(t,x,p_1,p_2),v^N(t,x,p_1,p_2))\geq \mathcal{H}_2(t,x,p_2,u^N(t,x,p_1,p_2),v). $$	
\end{enumerate}
Here $I$ stands for the identity matrix, whereas $$\mathcal{H}_i(t,x,p,u,v)\triangleq \langle p b(t,x,u,v)\rangle+h_i(t,x,u,v). $$ 

Notice that now  the system of Hamilton-Jacobi equation is the following system of parabolic equations:
\begin{equation}\label{eq:parabolic_system}
\begin{split}
\frac{\partial c_i}{\partial t}+\mathcal{H}_i(t,x,\nabla c_i,u^N(t,x,\nabla c_1,\nabla &c_2),u^N(t,x,\nabla c_1,\nabla c_2))\\&+\langle G(t,x)\nabla,\nabla\rangle c_i(t,x)=0.
\end{split}
\end{equation}
\begin{equation}\label{eq:parabolic_bound_condition}
c_i(T,x)=g_i(x).
\end{equation} This system is a variant of (\ref{eq:system_HJB2}) with $u^0(t,x)=u^N(t,x,\nabla c_1(t,x),\nabla c_2(t,x))$, $v^0(t,x)=v^N(t,x,\nabla c_1(t,x),\nabla c_2(t,x))$.

As it was mentioned above, we  consider strong generalized solutions of system~(\ref{eq:parabolic_system}). 
This solution concept relies on the following definitions. Let $\Upsilon$ be a subset of $[0,T]\times\rd$. Denote by $H^{1+\kappa}(\Upsilon)$ the set of  functions $\varphi:\Upsilon\rightarrow\mathbb{R}$ those satisfy H\"older condition for the exponent $\kappa$ with its derivatives w.r.t. spatial variables. Further, let $W^{1,2}_q(\Upsilon)$ be the set of functions $\varphi:\Upsilon\rightarrow\mathbb{R}$ such that $\varphi\in L_q(\Upsilon)$ and there exist generalized derivatives $\partial\varphi/\partial t$, $\partial\varphi/\partial x_i$, $\partial^2\varphi/\partial x_i\partial x_j$ which belong to $L_q(\Upsilon)$.

\begin{definition}\label{def:strong_solution_pde} The pair $(c_1,c_2)$ is a strong solution of (\ref{eq:parabolic_system}), (\ref{eq:parabolic_bound_condition}) if
	\begin{enumerate}
		\item $c_1,c_2\in L_{\infty}([0,T]\times \rd)$;
		\item for any bounded $\Xi\subset\rd$, and some $\kappa\in (0,1)$, $q>d+2$, $c_1,c_2\in H^{1+\kappa}([0,T]\times \mathrm{cl}(\Xi))\cap W^{1,2}_q((0,T)\times\Xi)$;
		\item (\ref{eq:parabolic_system}) holds almost everywhere in $(0,T)\times \Xi$, (\ref{eq:parabolic_bound_condition})  is fulfilled in $\Xi$.
	\end{enumerate}
\end{definition}

Note that any classical solution of (\ref{eq:parabolic_system}), (\ref{eq:parabolic_bound_condition}) is a strong solution. 

\begin{remark}\label{remark:2_order} There are several papers dealing with the link between the value functions of the stochastic differential games and strong solutions of the system of parabolic PDEs. Let us mention only \cite{Bensoussan_Frehse_2000}, \cite{Friedman}, \cite{Hamadene_Manucci}, \cite{Manucci_1}, \cite{Manucci_2}. In particular, for the case of stochastic differential games on the bounded domain of $\rd$ it is proved that if the functions $u^N$ and $v^N$ are continuous w.r.t. $p_i$, there exists a strong solution of the corresponding system of parabolic PDEs that provides the Nash equilibrium in the stochastic differential game \cite{Friedman_stoch_book} (see, also, \cite{Bensoussan_Frehse_1984}, \cite{Bensoussan_Frehse_2000}). The case when the strategies are not continuous  was studied in \cite{Manucci_2} under some additional assumptions. Several existence results for system (\ref{eq:parabolic_system}), (\ref{eq:parabolic_bound_condition}) were obtained in \cite{Hamadene_Manucci}. These results covers the cases when
	\begin{itemize}	
		\item the drift is bounded whereas $u^N$, $v^N$ are continuous w.r.t. adjoint variables; 
		\item the drifts and the running rewards of the players have a separate structure whereas the strategies $u^N$, $v^N$ are merely measurable.
	\end{itemize}
\end{remark}

\begin{theorem}\label{th:link_strong} Let the auxiliary stochastic system be given by (\ref{eq:stoch_sys}). Assume that conditions (A1)--(A4) hold. If $c_1,c_2$ is a strong generalized solution of (\ref{eq:parabolic_system}), (\ref{eq:parabolic_bound_condition}), then $(c_1,c_2)$ satisfies condition $(\mathcal{C})$ for generator $\Lambda$ given by (\ref{intro:generator_2nd_specific}).
\end{theorem}
\begin{proof} Let us fix $s,r\in [0,T]$, $s<r$. We put $(\widehat{\Omega}^{s,r},\widehat{\mathcal{F}}^{s,r},\{\widehat{\mathcal{F}}^{s,r}_t\}_{t\in [s,r]},\widehat{P}^{s,r})$ to be the standard probability space carrying the $d$-dimensional Wiener process on $[s,r]$. Note that here we assume that the probabilities $\widehat{P}^{s,r}$ do not depend on the initial state~$y$. Moreover, we set  $\overline{P}^{1,s,r}_{y,v}=\overline{P}^{2,s,r}_{y,u}\triangleq \widehat{P}^{s,r}$.
	
	Let $R>0$. Denote by $B_R$ the ball centered at the origin of the radius $R$. Since  $c_i\in H^{1+\kappa}([s,r]\times B_R)$, for each $i=1,2$, one can construct a function $c_{i,R}^\nabla:[0,T]\times\mathbb{R}^d\rightarrow\mathbb{R}^d$, that is H\"{o}lder continuous with the exponent $\kappa$ on $[0,T]\times\mathbb{R}^d$ and  coincides with $\nabla c_i$ on $[0,T]\times B_R$. Denote
	$$b_R^*(t,x)\triangleq b(t,x,u^N(t,x,c_{1,R}^\nabla(t,x), c_{2,R}^\nabla(t,x)),v^N(t,x, c_{1,R}^\nabla(t,x), c_{2,R}^\nabla(t,x))).$$
	Under condition (A1)--(A4) it follows from \cite{Veretennikov_strong_1} there exists  a stochastic process 
	$\widehat{Y}^{s,r}_{y,R} $ solving the stochastic differential equation
	$$d\widehat{Y}^{s,r}_{y,R}=b_R^*(t,\widehat{Y}^{s,r}_{y,R})dt+\sigma(t,\widehat{Y}^{s,r}_{y,R})dW_t,\ \ \widehat{Y}^{s,r}_{y,R}=y. $$ 	Put $$\eta^{s,r}_y\triangleq \delta_{u^N(t,c^\nabla_{1,R}(t,\widehat{Y}^{s,r}_{y,R}), c_{2,R}^\nabla(t,\widehat{Y}^{s,r}_{y,R}))}\otimes \delta_{v^N(t,c^\nabla_{1,R}(t,\widehat{Y}^{s,r}_{y,R}),c^\nabla_{2,R}(t,\widehat{Y}^{s,r}_{y,R}))}.$$ 
	We have that, for any $\phi\in \mathcal{D}$,
	\begin{equation}\label{martingale:2nd_order}
	\phi(\widehat{Y}^{s,r}_{y,R}(t))-\int_{s}^t\int_{U\times V}\Lambda_\tau[u,v]\phi(\widehat{Y}^{s,r}_{y,R}(t))\eta^{s,r}_{y,R}(\tau,d(u,v))d\tau 
	\end{equation} is a $\{\widehat{\mathcal{F}}^{s,r}_t\}_{t\in [s,r]}$-martingale. 
	Recall that for the considered case the generator $\Lambda$ is given by (\ref{intro:generator_2nd_specific}).
	
	If $\varrho\in [0, R]$, then denote by $\Theta^{s,r}_{\varrho,R}$  the exit time of $\widehat{Y}^{s,r}_{y,R}$ from $[s,r]\times B_\varrho$. Clearly, $\Theta^{s,r}_{\varrho,R}=\Theta_{\varrho,\varrho}$ and, for $t\in [s,\Theta^{s,r}_{\varrho,R}]$, 
	$$\widehat{Y}^{s,r}_{y,R}=\widehat{Y}^{s,r}_{y,\varrho}. $$ Further, let  $\widehat{Y}^{s,r}_y$ and $\eta^{s,r}_y$ be limits of $\widehat{Y}^{s,r}_{y,R}$ and $\eta^{s,r}_{y,R}$ respectively when $R\rightarrow\infty$. Taking the limit in (\ref{martingale:2nd_order}), we get that $(\widehat{\Omega}^{s,r},\widehat{\mathcal{F}}^{s,r},\{\widehat{\mathcal{F}}^{s,r}_t\}_{t\in [s,r]},\widehat{P}^{s,r},\widehat{Y}^{s,r}_{y}, \eta^{s,r}_{y})$ is an admissible control system for the generator $\Lambda$ given by (\ref{intro:generator_2nd_specific}).

	Let $\Theta^{s,r}_\varrho$ be the first exit time  of $\widehat{Y}^{s,r}_{y}$ from $[s,r]\times B_\varrho$. We have that, for $t\in [s,\Theta^{s,r}_\varrho]$ and any $R\geq\varrho$, $\widehat{Y}_{y}^{s,r}=\widehat{Y}^{s,r}_{y,R}$. Additionally, notice that
	$\Theta^{s,r}_\varrho$ converges to $r$ as $\varrho\rightarrow\infty$.
	
	By Ito-Krylov formula \cite[Theorem 2.10.1]{Krylov} we have that
	\begin{equation}\label{equal:expect_c_i_Ito_Krylov}
	\begin{split}
	\widehat{\mathbb{E}}^{s,r}\int_{s}^{\Theta^{s,r}_\varrho} &{}\left(\left[\frac{\partial }{\partial t}+\langle b^*(t,\widehat{Y}^{s,r}_y(t)),\nabla \rangle+\langle G(t,\widehat{Y}^{s,r}_y(t))\nabla,\nabla \right]c_i(t,\widehat{Y}^{s,r}_y(t))\right)dt\\ &=\widehat{\mathbb{E}}^{s,r} c_i(\Theta^{s,r}_\varrho,\widehat{Y}^{s,r}_y(\Theta^{s,r}_\varrho))-c_i(s,y). \end{split}
	\end{equation}
	Here $$b^*(t,x)\triangleq b(t,x,u^N(t,\nabla c_{1}(t,x),\nabla c_{2}(t,x)),v^N(t,\nabla c_{1}(t,x),\nabla c_{2}(t,x)))$$ is the limit of $b_R^*(t,x)$ when $R\rightarrow\infty$.
	
	Taking into account the definition of $\mathcal{H}_i$, we get that
	\begin{equation}\label{equal:H_i_h_i}
	\begin{split}
	\widehat{\mathbb{E}}^{s,r}\int_{s}^{\Theta^{s,r}_\varrho} \Bigl(\Bigl[\frac{\partial }{\partial t}+\langle b^*(t,\widehat{Y}^{s,r}_y(t)),\nabla \rangle+&\langle G(t,\widehat{Y}^{s,r}(t))\nabla,\nabla \Bigr]c_i(t,\widehat{Y}^{s,r}_y(t))\Bigr)dt\\=
	\widehat{\mathbb{E}}^{s,r}\int_{s}^{\Theta^{s,r}_\varrho} \Bigl[\mathcal{H}_i(t,\widehat{Y}^{s,r}_y(t),&\nabla c_1(t,\widehat{Y}^{s,r}_y(t)),\nabla c_2(t,\widehat{Y}^{s,r}_y(t)))\\ +&\langle G(t,\widehat{Y}^{s,r}_y(t))\nabla,\nabla c_i(t,\widehat{Y}^{s,r}_y(t))\rangle\Bigr]dt\\-\widehat{\mathbb{E}}^{s,r}\int_{s}^{\Theta^{s,r}_\varrho}\int_{U\times V} &h_i(t,\widehat{Y}^{s,r}_y(t),u,v)\eta^{s,r}_{y}(t,d(u,v))dt.
	\end{split}
	\end{equation} Since $(c_1,c_2)$ is a solution of (\ref{eq:parabolic_system}) and $\widehat{Y}^{r,s}_y$ has a density, (\ref{equal:expect_c_i_Ito_Krylov}), (\ref{equal:H_i_h_i}), we get
	\begin{equation}\label{equal:expect_c_i_outcome}
	\begin{split}
	\widehat{\mathbb{E}}^{s,r}\Bigl[c_i(\Theta^{s,r}_\varrho,\widehat{Y}^{s,r}_y(\Theta^{s,r}_\varrho))+\int_{s}^{\Theta^{s,r}_\varrho}\int_{U\times V} h_i(t,\widehat{Y}^{s,r}_y(t),u,v)\eta^{s,r}_{y}(t,d(&u,v))dt\Bigr]\\&=c_i(s,y).\end{split}
	\end{equation} 
	Since $\Theta^{s,r}_\varrho\rightarrow r$ as $\varrho\rightarrow\infty$,  we conclude that part (i) of Condition $(\mathcal{C})$ is fulfilled. 
	
	Parts (ii) and (iii) are proved in the same. Let us briefly describe the proof of part (ii). Fix $y\in\rd$ and $v\in V$. As above, we construct the processes $\mu^{s,r}_{y,v}$ and $\overline{Y}^{1,s,r}_{y,v}$ such that
	\begin{itemize}
		\item $\mu_{y,v}^{s,r}=\delta_{u^N(t,\overline{Y}^{1,s,r}_{y,v}(t),\nabla c_1(t,\overline{Y}^{1,s,r}_{y,v}(t)),\nabla c_2(t,\overline{Y}^{1,s,r}_{y,v}(t)))};$
		\item for any $\varphi\in\mathcal{D}$,
		$$\phi(\overline{Y}^{1,s,r}_{y,v}(t))-\int_{s}^t\int_U\Lambda_\tau[u,v]\phi(\overline{Y}^{1,s,r}_{y,v}(\tau))\mu_{y,v}^{s,r}(\tau,du) $$ is a $\{\widehat{\mathcal{F}}^{s,r}_t\}$-martingale.
	\end{itemize} 
	
	Let $\Theta^{1,s,r}_{y,v,\varrho}$ be the first exit time of $\overline{Y}^{1,s,r}_{y,v}$ from $[s,r]\times B_\varrho $. 	As above we use the Ito-Krylov formula \cite[Theorem 2.10.1]{Krylov} and get
	\begin{equation}\label{equal:expect_c_2_Ito_Krylov}
	\begin{split}
	\widehat{\mathbb{E}}^{s,r}\int_{s}^{\Theta^{s,r}_\varrho} &{}\left(\left[\frac{\partial }{\partial t}+\langle b^1_v(t,\widehat{Y}^{1,s,r}_y(t)),\nabla \rangle+\langle G(t,\widehat{Y}^{1,s,r}_y(t))\nabla,\nabla \right]c_2(t,\widehat{Y}^{s,r}_y(t))\right)dt\\ &=\widehat{\mathbb{E}}^{s,r} c_2(\Theta^{1,s,r}_\varrho,\widehat{Y}^{s,r}_y(\Theta^{1,s,r}_\varrho))-c_2(s,y).  \end{split}
	\end{equation} Here we denote $b^1_v(t,x)\triangleq b(t,x,u^N(t,x,\nabla c_1(t,x),\nabla c_2(t,x)),v)$. By condition (A4) using the same arguments as above, we get that
	\begin{equation*}\label{equal:expect_c_2_outcome}
	c_2(t,y)\geq \widehat{\mathbb{E}}^{s,r}\left[c_2(t,\overline{Y}^{s,r}_{y,v}(\Theta^{1,s,r}_\varrho))+\int_{s}^{\Theta^{1,s,r}_\varrho}\int_{U} h_2(t,\overline{Y}^{s,r}_{y,v}(t),u,v)\mu^{s,r}_{y,v}(t,du)dt\right].
	\end{equation*} 
	This proves part (ii) of Condition $(\mathcal{C})$.
\end{proof}

The following example is an illustration of Theorem \ref{th:link_strong}.
\begin{example}\normalfont Let $d=1$, $f(t,x,u,v)=\tilde{f}_1(t,x)u+\tilde{f}_2(t,x)v+\tilde{f}_3(t,x)$, $u,v\in [-1,1]$. Choose $h_1(t,x,u,v)\triangleq -\delta u^2/2$, $h_2(t,x,u,v)\triangleq -\delta v^2/2$. Finally, set $\sigma(t,x)\triangleq \delta$. In this case $u^N(t,x,p_1,p_2)=\lceil p_1 \tilde{f}_1(t,x)/\delta\rceil$, $v^N(t,x,p_1,p_2)=\lceil p_2 \tilde{f}_2(t,x)/\delta \rceil$, where we denote $$\lceil a \rceil \triangleq \left\{\begin{array}{cc}
	a, & |a|\leq 1 \\
	1, & a\geq 1, \\
	-1, & a\leq -1.
	\end{array}\right. $$ Thus,  (\ref{eq:parabolic_system}) takes the form
	\begin{equation}\label{eq:parabolic_example_1}
	\begin{split}
	\frac{\partial c_i}{\partial t}+\frac{\partial c_i}{\partial x}\Bigl(\Bigl\lceil \frac{\tilde{f}_1(t,x)}{\delta}\frac{\partial c_1}{\partial x} &{}\Bigr\rceil +\Bigl\lceil\frac{\tilde{f}_2(t,x)}{\delta}\frac{\partial c_2}{\partial x} \Bigr\rceil+ \tilde{f}_3(t,x) \Bigr)\\ &-\frac{\delta}{2}\Bigl\lceil\frac{\tilde{f}_i(t,x)}{\delta}\frac{\partial c_i}{\partial x} \Bigr\rceil^2+\frac{\delta^2}{2}\frac{\partial^2 c_i}{\partial x^2}=0,\ \ i=1,2.
	\end{split}
	\end{equation}
	This system has a strong solution \cite[Theorem 3.1]{Hamadene_Manucci}. Theorem \ref{th:near_Nash} provides that, given a solution of (\ref{eq:parabolic_example_1}) one can construct an approximate equilibrium for the original nonzero-sum differential game. However, system (\ref{eq:parabolic_example_1}) is highly nonlinear even in the simplest cases and can be solved only numerically. 
\end{example}

The following example shows that the class of functions satisfying condition $(\mathcal{C})$ is not limited by the solutions of the system of Hamilton-Jacobi PDEs.

\begin{example}\normalfont Consider the following controlled system 
	$$\dot{x}_1=u,\ \ \dot{x}_2=v,\ \ t\in [0,1], \ \ u,v\in [-1,1]. $$
	Assume that $\gamma_1(x_1,x_2)=\zeta x_1 - x_2$, $\gamma_2(x_1,x_2)=\zeta x_2 - x_1$. Here $\zeta\in (0,1/2)$. Now let us introduce the auxiliary stochastic system. Put $b=f$, $\sigma=\delta I$ (here $I$ stands for the identity matrix), $h_1=h_2=0$. We have that (\ref{eq:parabolic_system}) takes the form
	\begin{equation}\label{eq:system_example_2}
	\frac{\partial c_i}{\partial t}+\frac{\partial c_i}{\partial x_1}\cdot\mathrm{sgn}\Bigl(\frac{\partial c_1}{\partial x_1}\Bigr)+\frac{\partial c_i}{\partial x_2}\cdot\mathrm{sgn}\Bigl(\frac{\partial c_2}{\partial x_2}\Bigr)+\frac{\delta^2}{2}\Bigl(\frac{\partial^2 c_i}{\partial x_1^2}+\frac{\partial^2 c_i}{\partial x_2^2}\Bigr)=0.
	\end{equation}
	At the same time boundary condition (\ref{eq:parabolic_bound_condition})  takes the form
	\begin{equation}\label{eq:boundary_example_2}
	c_1(1,x_1,x_2)=\zeta x_1-x_2,\ \ c_2(1,x_1,x_2)=\zeta x_2-x_1. 
	\end{equation}
	It is easy to check that the functions $c_1^1(t,x_1,x_2)=\zeta x_1-x_2-(1-\zeta)(1-t)$, $c_2^1(t,x_1,x_2)=\zeta x_2-x_1-(1-\zeta)(1-t)$ solve system (\ref{eq:system_example_2}), (\ref{eq:boundary_example_2}).
	
	Now let us consider the smooth functions $c_1^2(t,x_1,x_2)=\zeta x_1-x_2+(1-\zeta)(1-t)$, $c_2^2(t,x_1,x_2)=\zeta x_2-x_1+(1-\zeta)(1-t)$. They do not solve system  (\ref{eq:system_example_2}), (\ref{eq:boundary_example_2}). However, the pair $(c_1^2,c_2^2)$ satisfies condition $(\mathcal{C})$. Indeed, the auxiliary stochastic system in the example is
	\begin{equation}\label{sys:example2}
	dY(t)=(u(t)+v(t))dt+\delta W(t). 
	\end{equation} Here $W(t)$ is the $2$-dimensional Wiener process. Let $s,r\in [0,1]$, $s<r$. Put $(\widehat{\Omega}^{s,r},\widehat{\mathcal{F}}^{s,r},\{\widehat{\mathcal{F}}^{s,r}_t\}_{t\in [s,r]},\widehat{P}^{s,r})$  be the standard probability space carrying the $2$-dimensional Wiener process on $[s,r]$. Further, we assume that $\widehat{P}^{s,r}$ does not depend on initial state $y$ and the probabilities $\overline{P}^{1,s,r}_{y,v}$, $\overline{P}^{2,s,r}_{y,u}$ are equal to $\widehat{P}^{s,r}$.
	
	To show that part (i) of Condition $(\mathcal{C})$ is fulfilled choose $y=(y_1,y_2)$ and pick $u=v=-1$. Thus, the $2$-dimensional process with the components $\widehat{Y}^{s,r}_{1,y}(t)=y_1-(t-s)+W_1(t-s)$, $\widehat{Y}^{s,r}_{2,y}(t)=y_2-(t-s)+W_2(t-s)$ satisfies (\ref{sys:example2}). Here $W_1(t)$, $W_2(t)$ are components of the Wiener processes $W(t)$. We have that
	\begin{equation*}
	\begin{split}
	\widehat{\mathbb{E}}^{s,r}c_i^2(r,\widehat{Y}^{s,r}_{1,y}(r),\widehat{Y}^{s,r}_{2,y}(r))&=\zeta y_i-y_{3-i}+(1-\varepsilon)(r-s)+(1-\varepsilon)(1-r)\\&=\zeta y_i-y_{3-i}+(1-\zeta)(1-s)=c_i^2(s,y_1,y_2). 
	\end{split}
	\end{equation*}
	Hence, part (i) of condition $(\mathcal{C})$ is fulfilled for $(c_1^2,c_2^2)$. Notice that in this case $\eta^{s,r}_y(t)=\delta_{-1}\otimes\delta_{-1}$. Now, let $v$ be an arbitrary element of $[-1,1]$. Pick $u=1$. We have that $\overline{Y}^{1,s,r}_{1,y,v}=y_1+(t-s)+W_1(t-s)$, $\overline{Y}^{1,s,r}_{2,y,v}=y_2+v(t-s)+W_2(t-s)$ satisfy (\ref{sys:example2}). We have that
	\begin{equation*}
	\begin{split}
	\widehat{\mathbb{E}}^{s,r}c_2^2(r,\overline{Y}^{1,s,r}_{1,y,v},\overline{Y}^{1,s,r}_{2,y,v})&=
	\zeta y_2-y_1+\zeta v(r-s)-(r-s)+(1-\zeta)(1-r)\\&=
	\zeta y_2-y_1+(1-\zeta)(1-s)-(r-s)(2-(v+1)\zeta)\\&\geq c_2^2(s,y_1,y_2).
	\end{split}
	\end{equation*} Thus, part (ii) of condition $(\mathcal{C})$ holds. Here $\mu^{s,r}_{y,v}(t)=\delta_{1}$. Analogously, picking $\overline{Y}^{2,s,r}_{1,y,u}=y_1+u(t-s)+W_1(t-s)$, $\overline{Y}^{2,s,r}_{2,y,u}=y_2+(t-s)+W_2(t-s)$, we get that 
	$\widehat{\mathbb{E}}^{s,r}c_1(r,\overline{Y}^{2,s,r}_{1,y,v},\overline{Y}^{2,s,r}_{2,y,v})\geq c_1(s,y_1,y_2)$. This proves part (iii) of condition $(\mathcal{C})$.
	
\end{example} 

\section{Limit of $\varepsilon$-equilibria}\label{sect:limit}

In this section we compare the result presented in the paper with the approach to Nash equilibria based on the framework of punishment strategies (see \cite{Chistyakov}, \cite{Cleimenov}, \cite{Kononenko},~\cite{Tolwinski}). The main purpose of this section is to prove that the limit of $\varepsilon$-Nash equilibria constructed by Theorem \ref{th:near_Nash} provides the Nash value in the class of punishment strategies.

Let $\mathrm{Val}_i$ denote the value function of the zero-sum differential game with the dynamics (\ref{eq:sys_original}) where the $i$-th player wishes to maximize her terminal payoff $\gamma_i(x(T))$ and the other player wishes to minimize this outcome. One can assume that the players use deterministic strategies with memory. In this case the strategy of the first (respectively, second) player is a family of functions from $[t_0,T]$ to $U$ $u_\bullet=\{u_{x(\cdot)}\}$ defined for all $x(\cdot)\in C([t_0,T],\rd)$ such that if $x(t)=y(t)$ for all $t\in [t_0,\theta]$, then 
$u_{x(\cdot)}(t)=u_{y(\cdot)}(t)$. Analogously, a strategy with memory of the second player is a family of functions $v_\bullet=\{v_{x(\cdot)}\}_{x(\cdot)\in C([t_0,T],\rd)}$ satisfying the property: the equality $x(t)=y(t)$ for all $t\in [t_0,\theta]$ implies that $v_{x(\cdot)}(t)=v_{y(\cdot)}(t)$ for $t\in [t_0,\theta]$. It suffices to consider only stepwise strategies. This means that there exists a partition of the time interval $[t_0,T]$ $\Delta=\{t_j\}_{j=0}^m$ such that if $x(t_j)=y(t_j)$, then $u_{x(\cdot)}(t)=u_{y(\cdot)}(t)$ (respectively, $v_{x(\cdot)}(t)=v_{y(\cdot)}(t)$). Denote the set of stepwise strategies with  memory on $[t_0,T]$ of the first (receptively, second) player by $\mathbb{U}[t_0]$ (respectively, $\mathbb{V}[t_0]$). Further, let $\mathcal{U}[t_0]$ (respectively, $\mathcal{V}[t_0]$) denote the set of all measurable functions $u:[t_0,T]\rightarrow U$ (respectively, $v:[t_0,T]\rightarrow V$). If $t_0\in [0,T]$, $x_0\in\rd$, $u_\bullet\in\mathbb{U}[t_0]$, $v\in \mathcal{V}[t_0]$, then denote by $x^1(\cdot,t_0,x_0,u_\bullet,v)$ a solution of the initial value problem:
\begin{equation}\label{eq:f_deterministic}
\frac{d}{dt}x(t)=f^1(t,x(t),u_{x(\cdot)}(t))+f^2(t,x(t),v(t)),\ \ x(t_0)=x_0. 
\end{equation} 
One can prove the existence and uniqueness result for (\ref{eq:f_deterministic}).

Analogously, for $u\in \mathcal{U}[t_0]$, $v_\bullet\in\mathbb{U}[t_0]$, let $x^2(\cdot,t_0,x_0,u,v_\bullet)$ solve the initial value problem:
$$\frac{d}{dt}x(t)=f^1(t,x(t),u(t))+f^2(t,x(t),v_{x(\cdot)}(t)),\ \ x(t_0)=x_0. $$

It follows from \cite{NN_PDG}, \cite{Subb_book} that
$$\mathrm{Val}_1(t_0,x_0)=\sup_{u_\bullet\in\mathbb{U}[t_0]}\inf_{v\in \mathcal{V}[t_0]}\gamma_1(x^1(T,t_0,x_0,u_\bullet,v)), $$
$$\mathrm{Val}_2(t_0,x_0)=\sup_{v_\bullet\in\mathbb{V}[t_0]}\inf_{u\in \mathcal{U}[t_0]}\gamma_2(x^2(T,t_0,x_0,u,v_\bullet)). $$
Moreover, \cite[Theorem 13.3]{Subb_book} implies that, for any  compact $G\subset \rd$, and $\varepsilon>0$, there exists  strategies $u^{G,\varepsilon}_\bullet\in\mathbb{U}[t_0]$, $v^{G,\varepsilon}_\bullet\in\mathbb{V}[t_0]$ such that, for any $x_0\in G$, $t_0\in [0,T]$,
\begin{equation*}
\mathrm{Val}_1(t_0,x_0)\geq \inf_{v\in\mathcal{V}[t_0]}\gamma_1(x^1(T,t_0,x_0,u^{G,\varepsilon}_\bullet,v))-\varepsilon, 
\end{equation*}
\begin{equation*}
\mathrm{Val}_2(t_0,x_0)\geq \inf_{u\in\mathcal{U}[t_0]}\gamma_2(x^2(T,t_0,x_0,u,v^{G,\varepsilon}_\bullet))-\varepsilon. 
\end{equation*}


Now, we turn to the nonzero-sum differential game with the dynamics given by~(\ref{eq:sys_original}) and players' payoffs determined by $\gamma_i(x(T))$. We consider Nash equilibria for this game within deterministic memory strategies as defined in  \cite{Cleimenov}, 
\cite{Tolwinski}. For the sake of shortness, we  recall only the characterization of Nash equilibrium payoffs. The precise definition of  equilibrium in the class of deterministic memory strategy and the proof of characterization theorem can be found in \cite{Chistyakov}, \cite{Cleimenov}, 
\cite{Tolwinski}.

\begin{definition}\label{def:Nash_Kleimenov}
	We say that $(a_1,a_2)\in\mathbb{R}^2$ is a Nash equilibrium payoff at $(t_0,x_0)$ if there exists a function $x(\cdot)\in C([t_0,T];\rd)$ solving 
	\begin{equation}\label{incl:original}
	\frac{d}{dt}x(t)=\mathrm{co}\{f_1(t,x(t),u)+f_2(t,x(t),v):u\in U,v\in V\},\ \ x(t_0)=x_0
	\end{equation}
	such that $a_i=\gamma_i(x(T))\geq \mathrm{Val}_i(t,x(t))$ for any $t\in [t_0,T]$.
\end{definition}
We denote the set of all Nash equilibrium payoffs at $(t_0,x_0)$ by $\mathcal{N}(t_0,x_0)$. Each set $\mathcal{N}(t_0,x_0)$ is closed.

\begin{theorem}\label{th:limit}
	For each natural $n$, let $\Lambda^n_t[u,v]$ be a generator of the L\'{e}vy-Khintchine type  \begin{multline*}(\Lambda_t^n[u,v]\phi)(x)\triangleq \frac{1}{2}\langle G^n(t,x,u,v)\nabla,\nabla\rangle\phi(x)+\langle b^n(t,x,u,v),\nabla\rangle\phi(x)\\+\int_{\rd}[\phi(x+y)-\phi(x)-\langle y,\nabla\phi(x)\rangle\mathbf{1}_{B_1}(y)]\nu^n(t,x,u,v,dy). \end{multline*} Additionally, let $h_i^n$, $i=1,2$, $n\in\mathbb{N}$, be a function from $[0,T]\times \rd\times U\times V$ to $\mathbb{R}$. Assume that $\Lambda^n_t[u,v]$ and $h_i^n$ satisfy conditions (L1)--(L8) for $\delta=\delta^n$. Further, let $(c_1^n,c_2^n)$ satisfy boundary condition $c_i^n(T,x)=\gamma_i(x)$ and Condition $(\mathcal{C})$ with the generator $\Lambda^n_t[u,v]$ and running payoffs $h_i^n$. If $$\delta^n\rightarrow 0,\ \ (c_1^n(t_0,x_0),c_2^n(t_0,x_0))\rightarrow (a_1,a_2)\text{ as }n\rightarrow\infty,$$
	then $(a_1,a_2)\in\mathrm{co}\mathcal{N}(t_0,x_0)$.
\end{theorem}
\begin{proof}
	Denote $\varepsilon^n\triangleq 2(RC+T)\delta^n$.  By Theorem \ref{th:near_Nash}, for each~$n$, there exists a public-signal correlated  profile of strategies $\mathfrak{w}^n=(\Omega^n,\mathcal{F}^n,\{\mathcal{F}^n\}_{t\in [t_0,T]},u^n_{x(\cdot)},v^n_{x(\cdot)}, P^n_{x(\cdot)})$ that is $\varepsilon^n$-equilibrium at $(t_0,x_0)$. Let $X^n(\cdot)$ and $P^n$ be generated by $\mathfrak{w}^n$ and initial position $(t_0,x_0)$. 
	
	Let us introduce the probability $\chi^n$ on $C([t_0,T];\rd)$ by the following rule:
	if $\mathcal{A}$ is a Borel subset of $C([t_0,T];\rd)$, then
	$$\chi^n(\mathcal{A})\triangleq P^n\{\omega\in\Omega^n:X^n(\cdot,\omega)\in \mathcal{A}\}. $$
	
	Further, let $\mathcal{S}(t_0,x_0)$ denote the set of solution of (\ref{incl:original}).
	Notice that
	\begin{itemize}
		\item $\mathrm{supp}(\chi^n)\subset \mathcal{S}(t_0,x_0)$;
		\item any $x(\cdot)$ in $\mathcal{S}(t_0,x_0)$ is Lipschitz continuous with the constant $M$ and satisfies the initial condition $x(t_0)=x_0$.
	\end{itemize}   Thus, by \cite{Billingsley} we get that $\{\chi^n\}$ is relatively compact with respect to the narrow convergence. Without loss of generality, we can assume that there exists a probability on $C([t_0,T];\rd)$ $\chi$ such that
	$\{\chi^n\}_{n=1}^\infty$ converges narrowly to $\chi$. Note that each motion $x(\cdot)$ from $\mathrm{supp}(\chi)$ satisfies (\ref{incl:original}).
	
	Further, by construction of $\chi^n$ we have that
	\begin{equation}\label{equal:E_n_chi_n}
	\mathbb{E}^n\gamma_i(X^n(T))=\int_{C([t_0,T];\rd)}\gamma_i(x(T))\chi^n(d(x(\cdot))).
	\end{equation} Here $\mathbb{E}^n$ denotes the expectation according to the probability $P^n$.
	By Theorem \ref{th:near_Nash} and (\ref{equal:E_n_chi_n}) we get
	$$\left|\int_{C([t_0,T];\rd)}\gamma_i(x(T))\chi^n(d(x(\cdot)))-c_i^n(t_0,x_0)\right|\leq \varepsilon^n. $$ Passing to the limit when $n\rightarrow\infty$,  we conclude that
	\begin{equation}\label{equal:limit_a}
	\int_{C([t_0,T];\rd)}\gamma_i(x(T))\chi(d(x(\cdot)))=a_i. 
	\end{equation}
	
	For $a\in\mathbb{R}$, let $a^+$ denote $a\vee 0$. Now, we shall prove that 
	\begin{equation}\label{equal:Val_gamma_null}
	\int_{C([t_0,T];\rd)}\int_{t_0}^T[\mathrm{Val}_i(t,x(t))-\gamma_i(x(T))]^+dt \chi(d(x(\cdot)))=0.
	\end{equation}

	First, we claim that, for any $\tau\in [t_0,T]$,
	\begin{equation}\label{equal:Val_tau}
	\int_{C([t_0,T];\rd)}[\mathrm{Val}_i(\tau,x(\tau))-\gamma_i(x(T))]^+\chi(d(x(\cdot)))=0.
	\end{equation}
	Indeed, assume for definiteness  that 
	\begin{equation*}
	\int_{C([t_0,T];\rd)}[\mathrm{Val}_1(\tau,x(\tau))-\gamma_1(x(T))]^+\chi(d(x(\cdot)))=3\rho>0.
	\end{equation*} 
	Since the function $x(\cdot)\mapsto[\mathrm{Val}_1(\tau,x(\tau))-\gamma_1(x(T))]^+$ is continuous, and the probabilities $\chi$, $\chi^n$ are concentrated on the compact $\mathcal{S}(t_0,x_0)$, there exists a number $N_0$ such that, for all $n>N_0$,
	\begin{equation}\label{ineq:int_val_n}
	\int_{C([0,T];\rd)}[\mathrm{Val}_1(\tau,x(\tau))-\gamma_1(x(T))]^+\chi^n(d(x(\cdot)))\geq 2\rho>0.
	\end{equation}  
	Moreover, one can assume that, for $n>N_0$,
	\begin{equation}\label{ineq:choice_n}
	\varepsilon^n<\rho.
	\end{equation}
	Further, notice that 
	\begin{equation*}\begin{split}\int_{C([0,T];\rd)}[\mathrm{Val}_1(\tau,x(\tau))&-\gamma_1(x(T))]^+\chi^n(d(x(\cdot)))\\=&\mathbb{E}^n[\mathrm{Val}_1(\tau,X^n(\tau))-\gamma_i(X^n(T))]^+.\end{split}\end{equation*} Denote by $\Xi^n_{\tau}$ the set of $\omega\in\Omega^n$ such that $$\mathrm{Val}_1(\tau,X^n(\tau,\omega))\geq\gamma_1(X^n(T,\omega)). $$
	
	Let $\mathfrak{w}^{n,\tau,\rho}=(\Omega^{n},\mathcal{F}^{n}, \{\mathcal{F}^{n}_t\}_{t\in [t_0,T]},u^{n,\tau,\rho}_{x(\cdot)},v^{n}_{x(\cdot)}, P^{n}_{x(\cdot)})$ be the deviation by the first player from $\mathfrak{w}^n$ with
	$$u^{n,\tau,\rho}_{x(\cdot)}(t,\omega)\triangleq \left\{\begin{array}{cc}
	u^n_{x(\cdot)}(t,\omega), & t\in [t_0,T], \omega\notin\Xi^n_{\tau}, \\
	& t\in [t_0,\tau), \omega\in\Xi^n_{\tau}, \\
	u^{G(\tau),\rho}_{x(\cdot)}, & t\in [\tau,T], \omega\in \Xi^n_{\tau}.
	\end{array}\right. $$
	Here $G(\tau)=\{x(\tau): x(\cdot)\in\mathcal{S}(t_0,x_0)\},$ and $u^{G(\tau),\rho}_\bullet$ is the  first player's strategy  that is $\rho$-optimal for the zero-sum differential game with the dynamics given by (\ref{eq:sys_original}) and objective function $\gamma_1(x(T))$ at any position from $G(\tau)$.
	
	Let $X^{n,\tau,\rho}(\cdot)$ and $P^{n,\tau,\rho}$ be generated by $\mathfrak{w}^{n,\tau,\rho}$ and $(t_0,x_0)$. Furthermore, let $\mathbb{E}^{n,\tau,\rho}$ be the expectation according to $P^{n,\tau,\rho}$. We have that, for $\omega\in \Omega^n\setminus\Xi^n_\tau$,  $X^{n,\tau,\rho}(\cdot,\omega)=X^n(\cdot,\omega)$. Moreover if $\omega\in\Xi^n_\tau$, then $X^{n,\tau,\rho}(t,\omega)=X^n(t,\omega)$ for $t\in [t_0,\tau]$ and $\gamma_1(X^{n,\tau,\rho}(T,\omega))\geq \mathrm{Val}_1(\tau,X^n(\tau))-\rho$. This and (\ref{ineq:int_val_n}) imply the inequality
	\begin{multline*}
	\mathbb{E}^{n,\tau,\rho}\gamma_1(X^{n,\tau,\rho}(T))= \mathbb{E}^{n,\tau,\rho}\gamma_1(X^{n,\tau,\rho}(T))\mathbf{1}_{\Omega^n\setminus \Xi^n_\tau}+\mathbb{E}^{n,\tau,\rho}\gamma_1(X^{n,\tau,\rho}(T))\mathbf{1}_{\Xi^n_\tau}\\\geq \mathbb{E}^{n}\gamma_1(X^{n}(T))\mathbf{1}_{\Omega^n\setminus \Xi^n_\tau}+\mathbb{E}^{n}\mathrm{Val}_1(\tau,X^{n}(\tau))\mathbf{1}_{\Xi^n_\tau}- \rho\\\geq \mathbb{E}^n\gamma_1(X^{n}(T))+2\rho-\rho.
	\end{multline*}
	Since $\varepsilon^n<\rho$ (see (\ref{ineq:choice_n})), we get that $\mathfrak{w}^n$ is not a public-signal correlated $\varepsilon^n$-equilibrium. This contradicts with the choice of $\mathfrak{w}^n$. Thus, (\ref{equal:Val_tau}) is fulfilled. 
	
	Further, let $m$ be a natural number, for $k=0,\ldots, m$, and let $\tau_k^m\triangleq t_0+(T-t_0)k/m$. Recall that, for any $x(\cdot)\in\mathcal{S}(t_0,x_0)$,
	$$ \|x(t')-x(t'')\|\leq M|t'-t''|.$$ Additionally, the function $\mathrm{Val}_i$ is Lipschitz continuous. Thus, there exists a constant~$C_1$ such that, for any $x(\cdot)\in\mathcal{S}(t_0,x_0)$, any $t',t''\in [0,T]$,
	$$|\mathrm{Val}_i(t',x(t'))-\mathrm{Val}_i(t'',x(t''))|\leq C_1|t'-t''|. $$ Since $\chi$ is concentrated on $\mathcal{S}(t_0,x_0)$, we have that
	\begin{multline*}
	\int_{C([t_0,T];\rd)}\int_{t_0}^T[\mathrm{Val}_i(t,x(t))-\gamma_i(x(T))]^+dt\chi(d(x(\cdot)))\\=
	\sum_{k=1}^m\int_{C([t_0,T];\rd)}\int_{\tau_{k-1}^m}^{\tau_k^m}
	[\mathrm{Val}_i(t,x(t))-\gamma_i(x(T))]^+dt\chi(d(x(\cdot)))\\
	\leq \sum_{k=1}^m\int_{C([t_0,T];\rd)}\int_{\tau_{k-1}^m}^{\tau_k^m}
	[\mathrm{Val}_i(\tau_{k-1}^m, x(\tau_{k-1}^m))-\gamma_i(x(T))]^+dt\chi(d(x(\cdot)))\\
	+\sum_{k=1}^m\int_{C([t_0,T];\rd)}\int_{\tau_{k-1}^m}^{\tau_k^m}|\mathrm{Val}_i(t,x(t))- \mathrm{Val}_i(\tau_{k-1}^m,x(\tau_{k-1}^m))|dt\chi(d(x(\cdot)))\\=
	\frac{(T-t_0)}{m}\sum_{k=1}^m\int_{C([t_0,T];\rd)}[\mathrm{Val}_i(\tau_{k-1}^m,x(\tau_{k-1}^m))- \gamma_i(x(T))]^+\chi(d(x(\cdot)))\\+C_1\frac{(T-t_0)}{2m}.
	\end{multline*} Using (\ref{equal:Val_tau}), we get that
	\begin{equation*}
	0\leq\int_{C([0,T];\rd)}\int_{t_0}^T[\mathrm{Val}_i(t,x(t))-\gamma_i(x(T))]^+dt\chi(d(x(\cdot))) \leq C_1\frac{(T-t_0)}{2m}.
	\end{equation*} Passing to the limit when $m\rightarrow\infty$, we get (\ref{equal:Val_gamma_null}).
	
	This implies that, for $\chi$-a.e. $x(\cdot)$,
	$(\gamma_1(x(T)),\gamma_2(x(T)))$ is a Nash equilibrium value. Using (\ref{equal:limit_a}), we obtain the conclusion of the Theorem.
\end{proof}

\section{Construction of near equilibrium strategies}\label{sect:construction}
The aim of this section is to construct a profile of public-signal correlated strategies $\mathfrak{w}^*$ for a given initial position $(t_0,x_0)\in [0,T]\times\rd$ and any partition of the time interval $[t_0,T]$. Below (see Sections \ref{sect:extremal} and \ref{sect:proof_of_the_main_result}) we show that $\mathfrak{w}^*$ is an approximate equilibrium at $(t_0,x_0)$. This will prove Theorem \ref{th:near_Nash}.

First, for $t\in [0,T]$, $x,y\in \rd$, let $u_\natural,u^\natural:[0,T]\times \rd\times\rd\rightarrow U$, $v_\natural,v^\natural:[0,T]\times \rd\times\rd\rightarrow V$ be measurable functions satisfying
\begin{equation}\label{intro:u_nat_down}
u_\natural(t,x,y)\in \mathrm{Argmin}\{\langle x-y,f_1(t,x,u) \rangle:u\in U\}, 
\end{equation}
\begin{equation}\label{intro:u_nat_up}u^\natural(t,x,y)\in \mathrm{Argmax}\{\langle x-y,f_1(t,x,u) \rangle:u\in U\}, \end{equation}
\begin{equation}\label{intro:v_nat_down}v_\natural(t,x,y)\in \mathrm{Argmin}\{\langle x-y,f_2(t,x,v) \rangle:v\in V\}, \end{equation}
\begin{equation*}\label{intro:v_nat_up}v^\natural(t,x,y)\in \mathrm{Argmax}\{\langle x-y,f_2(t,x,v) \rangle:v\in V\}. \end{equation*}
The existence of these functions follows from the Kuratowski--Ryll-Nardzewski selection theorem \cite[Theorem 18.13]{Infinite_dimensional_analysis}.

For $\theta\in\mathbb{R}$,  put
\begin{equation}\label{intro:alpha_3}
\tilde{\alpha}(\theta)\triangleq  \frac{4}{3}M\sqrt{1+(M)^2}e^{T/2}\sqrt{|\theta|},
\end{equation}
\begin{equation}\label{intro:epsilon}
\epsilon(\theta)\triangleq
2(M)^2\theta+2\tilde{\alpha}(\theta)+2(\alpha(\theta))^2+(KM)^2\theta^2 +(K)^2\tilde{\alpha}(\theta)\theta+(K)^2\theta.
\end{equation}

Recall that the pair of functions $(c_1,c_2)$ satisfies Condition $(\mathcal{C})$ (see Definition~\ref{def:cond_C}).

Let $(t_0,x_0)\in [0,T]\times\rd$ be an initial position, and let $\Delta=\{t_k\}_{k=0}^n$ be a partition of the time interval $[t_0,T]$. The position $(t_0,x_0)$ and the partition $\Delta$ are parameters for the profile of strategies $\mathfrak{w}^*$ defined below. However, we do not indicate the dependence of the profile of strategies on them assuming that $(t_0,x_0)$ and $\Delta$ are fixed.

Denote by $d(\Delta)$ the fineness of $\Delta$. Let the filtered measurable space $(\widehat{\Omega}^{t_{k-1},t_{k}},\widehat{\mathcal{F}}^{t_{k-1},t_{k}}, \{\widehat{\mathcal{F}}^{t_{k-1},t_{k}}_t\}_{t\in [t_{k-1},t_{k}]})$, families of probabilities $\widehat{P}^{t_{k-1},t_{k}}_y$, $\overline{P}^{1,t_{k-1},t_{k}}_{y,v}$, $\overline{P}^{2,t_{k-1},t_{k}}_{y,u}$ and families of stochastic processes $\widehat{Y}^{t_{k-1},t_{k}}_y$, $\overline{Y}^{1,t_{k-1},t_{k}}_{y,v}$, $\overline{Y}^{2,t_{k-1},t_{k}}_{y,u}$, $\eta^{t_{k-1},t_{k}}_y$, $\mu^{1,t_{k-1},t_{k}}_{y,v}$, $\nu^{2,t_{k-1},t_{k}}_{y,u}$ be chosen according to Condition~$(\mathcal{C})$ for $s=t_{k-1}$, $r=t_k$. Below, $\widehat{\mathbb{E}}^{t_{k-1},t_{k}}_y$, $\overline{\mathbb{E}}^{1,t_{k-1},t_{k}}_{y,v}$, $\overline{\mathbb{E}}^{2,t_{k-1},t_{k}}_{y,u}$ state for the expectations according to the probabilities $\widehat{P}^{t_{k-1},t_{k}}_y$, $\overline{P}^{1,t_{k-1},t_{k}}_{y,v}$, $\overline{P}^{2,t_{k-1},t_{k}}_{y,u}$ respectively.

Denote 
\begin{equation}\label{intro:Psi}
\Psi^k(x,y)\triangleq \widehat{\mathbb{E}}_y^{t_{k-1},t_{k}} \left\|x-\widehat{Y}^{t_{k-1},t_{k}}_y(t_{k})\right\|^2. 
\end{equation}

Now let us define the public-signal correlated profile of strategies $\mathfrak{w}^*=(\Omega^*,\mathcal{F}^*,\{\mathcal{F}^*_t\}_{t\in [t_0,T]},u^*_{x(\cdot)},v^*_{x(\cdot)},P^*_{x(\cdot)})$ by the following rules.

First, set $$\Omega^*\triangleq \btimes_{k=1}^n(\widehat{\Omega}^{t_{k-1},t_k})^3,$$ $$ \mathcal{F}^*\triangleq\botimes_{k=1}^n(\widehat{\mathcal{F}}^{t_{k-1},t_k}\otimes\widehat{\mathcal{F}}^{t_{k-1},t_k}\otimes\widehat{\mathcal{F}}^{t_{k-1},t_k}).$$
If $t\in [t_{k-1},t_k]$, then put
\begin{multline*}\mathcal{F}_t^*\triangleq \left[\botimes_{j=1}^{k-1}\left(\widehat{\mathcal{F}}^{t_{j-1},t_j}_{t_j}\otimes \widehat{\mathcal{F}}^{t_{j-1},t_j}_{t_j}\otimes\widehat{\mathcal{F}}^{t_{j-1},t_j}_{t_j}\right)\right]
\otimes \left(\widehat{\mathcal{F}}^{t_{k-1},t_k}_{t}\otimes \widehat{\mathcal{F}}^{t_{k-1},t_k}_{t}\otimes\widehat{\mathcal{F}}^{t_{k-1},t_k}_{t}\right) \\\otimes\left[\botimes_{j=k+1}^n \left(\widehat{\mathcal{F}}^{t_{j-1},t_j}_{t_{j-1}} \otimes \widehat{\mathcal{F}}^{t_{j-1},t_j}_{t_{j-1}} \otimes\widehat{\mathcal{F}}^{t_{j-1},t_j}_{t_{j-1}}\right) \right].\end{multline*}

Notice that the elements of $\Omega^*$ are $3n$-tuples $\omega=(\omega_1^0,\omega_1^1,\omega_1^2,\ldots,\omega_n^0,\omega_n^1,\omega_n^2)$. Informally speaking, $(\omega_1^0,\ldots,\omega_n^0)$ corresponds to the case when both players use the Nash equilibrium strategies; $(\omega_1^1,\ldots,\omega_n^1)$ is used when the second players deviates, whereas $(\omega_1^2,\ldots,\omega_n^2)$ works when the first player changes her strategy.


To define the probability $P_{x(\cdot)}$, and the processes $u_{x(\cdot)}$, $v_{x(\cdot)}$ let us introduce auxiliary continuous-time stochastic processes $Y^0,Y^1_{x(\cdot)},Y^2_{x(\cdot)}$  by the following rules:
\begin{itemize}
	\item $Y^0(t_0)=Y^1_{x(\cdot)}(t_0)=Y^2_{x(\cdot)}(t_0)\triangleq x_0$;
	\item if $Y^0(t)$, $Y^1_{x(\cdot)}(t)$ and $Y^2_{x(\cdot)}(t)$  are already defined on $[t_0,t_{k-1}]$, then set, for $t\in [t_{k-1},t_k]$, \begin{equation}\label{intro:Y_0}
	Y^0(t)\triangleq \widehat{Y}^{t_{k-1},t_k}_{Y^0(t_{k-1})}(t);
	\end{equation} 
	\item if, for all $j=1,\ldots, k$,
	\begin{equation}\label{ineq:decision}
	\begin{split}
	\Psi^j(x(t_j),Y^0(t_{j-1}))\leq  \|x(t_{j-1})-Y^0&(t_{j-1})\|^2(1+\beta(t_{j}-t_{j-1}))\\&+(4\delta^2+\epsilon(t_{j}-t_{j-1}))
	\cdot(t_{j}-t_{j-1}),
	\end{split}
	\end{equation} then, for $t\in [t_{k-1},t_k]$, put $Y^1_{x(\cdot)}(t)=Y^2_{x(\cdot)}(t)\triangleq Y^0(t)$;
	\item if inequality (\ref{ineq:decision}) violates for some $j=1,\ldots, k$, then define the processes $Y^1_{x(\cdot)}(t)$, $Y^2_{x(\cdot)}(t)$ on $[t_{k-1},t_k]$ by the rule:
	$$Y^1_{x(\cdot)}(t)\triangleq \overline{Y}^{1,t_{k-1},t_k}_{Y^1_{x(\cdot)}(t_{k-1}),v^\natural(t_{k-1},x(t_{k-1}),Y^1_{x(\cdot)}(t_{k-1}))}(t),$$ 
	$$Y^2_{x(\cdot)}(t)\triangleq \overline{Y}^{2,t_{k-1},t_k}_{Y^2_{x(\cdot)}(t_{k-1}),u^\natural(t_{k-1},x(t_{k-1}),Y^2_{x(\cdot)}(t_{k-1}))}(t).$$  
\end{itemize}

For $t\in [t_{k-1},t_{k})$, set
$$u^*_{x(\cdot)}(t)\triangleq u_\natural(t_{k-1},x(t_{k-1}),Y^1_{x(\cdot)}(t_{k-1})), $$
\begin{equation}\label{intro:v_star}
v^*_{x(\cdot)}(t)\triangleq v_\natural(t_{k-1},x(t_{k-1}),Y^2_{x(\cdot)}(t_{k-1})). 
\end{equation}

To complete the definition of the profile of strategies $\mathfrak{w}^*$ it remains to introduce the probability $P^*$. To this end,  we let us define the  sequences of auxiliary $\sigma$-algebra $\{\mathcal{G}_k\}_{k=0}^n$ and auxiliary probabilities $\{P_{k,x_{(\cdot)}}\}_{k=0}^n$ by the following rules.   

Put
\begin{equation*}\mathcal{G}_{k}\triangleq \left[\botimes_{j=1}^{k}(\widehat{\mathcal{F}}^{t_{j-1},t_j}_{t_j}\otimes \widehat{\mathcal{F}}^{t_{j-1},t_j}_{t_j}\otimes\widehat{\mathcal{F}}^{t_{j-1},t_j}_{t_j})\right]
\\ \otimes\left\{\varnothing,\Gamma_{k} \right\},\end{equation*}
where
$$\Gamma_{k}\triangleq \btimes_{j=k+1}^n(\widehat{\Omega}^{t_{j-1},t_j})^3. $$ Formally, we assume that $\mathcal{G}_0=\{\varnothing,\Omega^*\}$. 

One can check that the random variables $Y^0(t_k)$ and $Y^i_{x(\cdot)}(t_k)$, $i=1,2$, are $\mathcal{G}_k$ measurable. Moreover,
$\mathcal{G}_k\subset\mathcal{F}^*_{t_{k}}$ and $\mathcal{G}_n=\mathcal{F}^*_{T}=\mathcal{F}^*$

The probability $P_{0,x(\cdot)}$ is  defined on $\mathcal{G}_0$ in the trivial way.

Further, assume that $P_{k-1,x(\cdot)}$  is already defined. Note that there exists a function  $p_{k,x(\cdot)}:(\widehat{\mathcal{F}}^{t_{k-1},t_k}\otimes\widehat{\mathcal{F}}^{t_{k-1},t_k}\otimes\widehat{
	\mathcal{F}}^{t_{k-1},t_k})\times \Omega^*\rightarrow [0,1]$ satisfying by the following properties:
\begin{itemize}
	\item for any $\omega\in \Omega^*$,  $p_{k,x(\cdot)}(\cdot,\omega)$ is a probability on $\widehat{\mathcal{F}}^{t_{k-1},t_k}\otimes\widehat{\mathcal{F}}^{t_{k-1},t_k}\otimes\widehat{\mathcal{F}}^{t_{k-1},t_k}$;
	\item for any $B\in \widehat{\mathcal{F}}^{t_{k-1},t_k}\otimes\widehat{\mathcal{F}}^{t_{k-1},t_k}\otimes\widehat{
		\mathcal{F}}^{t_{k-1},t_k}$, the function $\omega\mapsto p_{k,x(\cdot)}(B,\omega)$ is $\mathcal{G}_{k-1}$-measurable;
	\item if $B^0,B^1,B^2\in\widehat{\mathcal{F}}^{t_{k-1},t_k}$, $\omega\in\Omega^*$, then \begin{equation*}\begin{split}
	p_{k,x(\cdot)}(B^0\times B^1\times &B^2,\omega)\\=&\widehat{P}_{Y^0(t_{k-1},\omega)}^{t_{k-1},t_{k}}(B^0)
	\widehat{P}_{Y^1(t_{k-1},\omega)}^{t_{k-1},t_{k}}(B^1) \widehat{P}_{Y^2(t_{k-1},\omega)}^{t_{k-1},t_{k}}(B^2)
	\end{split}
	\end{equation*} in the case when (\ref{ineq:decision}) is fulfilled for all $j=1,\ldots,k$ and \begin{equation*}\begin{split}
	p_{k,x(\cdot)}(B^0\times B^1\times B^2,\omega&)\\=\widehat{P}_{Y^0(t_{k-1},\omega)}^{t_{k-1},t_{k}}(B^0) &\cdot\overline{P}_{Y^1(t_{k-1},\omega),v^\natural(t_{k-1},x(t_{k-1}),Y^1(t_{k-1},\omega))}^{1,t_{k-1},t_{k}}(B^1)\\
	&\cdot\overline{P}_{Y^2(t_{k-1},\omega),u^\natural(t_{k-1},x(t_{k-1}),Y^2(t_{k-1},\omega))}^{2,t_{k-1},t_{k}}(B^2)\end{split}
	\end{equation*} when (\ref{ineq:decision}) violates for some $j=1,\ldots,k$.
\end{itemize}

For $$A\in \botimes_{j=1}^{k-1}\left(\widehat{\mathcal{F}}^{t_{j-1},t_j}\otimes\widehat{\mathcal{F}}^{t_{j-1},t_j}\otimes\widehat{
	\mathcal{F}}^{t_{j-1},t_j}\right),\ \ B\in \widehat{\mathcal{F}}^{t_{k-1},t_k}\otimes \widehat{\mathcal{F}}^{t_{k-1},t_k}\otimes\widehat{
	\mathcal{F}}^{t_{k-1},t_k},$$ put
$$P_{k,x(\cdot)}\left(A\times B\times \Gamma_{k}\right)\\\triangleq \int_{A\times  \Gamma_{k-1}}p_{k,x(\cdot)}(B,\omega)P_{k-1,x(\cdot)}(d\omega). $$

$P_{k,x(\cdot)}$ is extended to the whole $\sigma$-algebra $\mathcal{G}_k$ in the standard way.	
It is easy to check that the restriction of $P_{k,x(\cdot)}$ on $\mathcal{G}_{k-1}$ coincides with $P_{k-1,x(\cdot)}$.

To complete the definition of the probability $P_{x(\cdot)}^*$ observe that $\mathcal{G}_n=\mathcal{F}^*$ and set
$$P_{x(\cdot)}^*\triangleq P_{n,x(\cdot)} .$$

Let us clarify the meaning of the processes $Y^0$, $Y_{x(\cdot)}^1$, $Y_{x(\cdot)}^2$. They play the role of models of the game. The process $Y^0$ is used when the both players behave according to $\mathfrak{w}^*$, whereas $Y_{x(\cdot)}^1$ (respectively, $Y^2_{x(\cdot)}$) works when the second (respectively, first) player deviates.

\section{Properties of the models of the game}\label{sect:extremal}
First, let us consider the case when both players form their control according to the profile of strategies $\mathfrak{w}^*$.
Let $X^*(\cdot)$, $P^*$ be generated by $\mathfrak{w}^*$ and $(t_0,x_0)$. Denote by $\mathbb{E}^*$ the expectation according to~$P^*$. For $i=1,2$, let $\mathcal{Y}^{*,i}(t)\triangleq Y^i_{X^*(\cdot)}(t)$. Further, denote by $\eta_k^0$ the stochastic process with values in $\mathrm{rpm}(U\times V)$ defined by the rule:
\begin{equation}\label{intro:eta_0}
\eta_k^0(t)\triangleq \eta^{t_{k-1},t_k}_{Y^0(t_{k-1})}(t). 
\end{equation}

Notice that, for any $\phi\in\mathcal{D},$ the process 
\begin{equation}\label{express:Y_0_martingale}
\phi(Y^0(t))-\int_{t_{k-1}}^t\Lambda_\tau[u,v]\phi(Y^0(\tau))\eta^0_k(\tau,d(u,v))d\tau
\end{equation} is a $\{\mathcal{F}^*_t\}_{t\in [t_{k-1},t_k]}$-martingale.

\begin{lemma}\label{lm:key_estimate}
	The following statements hold true: 
	\begin{enumerate}
		\item $Y^0(t)=\mathcal{Y}^{*,1}(t)=\mathcal{Y}^{*,2}(t)$ for $t\in [t_0,T]$; 
		\item for $k=1,\ldots,n$,
		\begin{equation}\label{ineq:E_star_cond}
		\begin{split}
		\mathbb{E}^*(\|X^*(t_k)-Y^0(t_k)\|^2&|\mathcal{F}^*_{t_{k-1}})\\\leq \|X^*(t_{k-1})-Y^0&(t_{k-1})\|^2(1+\beta(t_k-t_{k-1}))\\+(4\delta^2&+\epsilon(t_k-t_{k-1}))
		\cdot(t_k-t_{k-1}).\end{split}
		\end{equation}
	\end{enumerate}
\end{lemma}
\begin{proof}
	The proof is close to the proof of \cite[Lemma 14]{averboukh_SIAM}.

	First, assume that, for $j=0,\ldots,k-1$,
	\begin{equation}\label{equal:Y_t_K_minus}
	Y^0(t_{j})=\mathcal{Y}^{*,1}(t_{j})=\mathcal{Y}^{*,2}(t_{j})
	\end{equation}
	Thus, on each time interval $[t_{k-1},t_k]$ the process $X^*$ is $\mathcal{F}^*_{t_{k-1}}$-measurable. Notice that, for $t\in [t_{k-1},t_k]$,  the following condition holds in the almost sure sense:
	\begin{equation*}\label{eq:x_deterministc}
	\frac{d}{dt}X^*(t)=f(t,X^*(t),u_\natural(t,X^*(t_{k-1}),Y^0(t_{k-1})),v_\natural(t,X^*(t_{k-1}),Y^0(t_{k-1}))).
	\end{equation*} 
	
	We have that, for $t\in [t_{k-1},t_k]$,
	\begin{equation}\label{ineq:x_t_x_s}
	\|X^*(t)-X^*(t_{k-1})\|\leq M(t-t_{k-1})\ \ P^*\text{-a.s}.
	\end{equation}  Assuming that $\delta\leq 1$, using that (\ref{express:Y_0_martingale}) is $\{\mathcal{F}^*_t\}_{t\in[t_{k-1},t_k]}$-martingale for any $\phi\in\mathcal{D}$, and conditions (L5), (L8), one can prove that
	\begin{equation}\label{ineq:Y_t_Y_s}
	{\mathbb{E}}^*(\|Y^0(t)-Y^0(t_{k-1})\|^2|\mathcal{F}^*_{t_{k-1}})\leq \delta^2(t-t_{k-1})+\tilde{\alpha}(t-t_{k-1})\cdot (t-t_{k-1}),
	\end{equation} where $\tilde{\alpha}(\cdot)$ is defined by (\ref{intro:alpha_3}). 
	The proof of (\ref{ineq:Y_t_Y_s}) is similar to the proof of \cite[Lemma~13]{averboukh_SIAM}.
	
	Since $\|X^*(t_k)-Y^{0}(t_k)\|^2=\|(X^*(t_k)-X^*(t_{k-1}))- (Y^{0}(t_{k})-Y^{0}(t_{k-1}))+(X^*(t_{k-1})-Y^0(t_{k-1}))\|^2$, we have that
	\begin{equation}\label{equal:x_0_r_Y_r}
	\begin{split}
	\mathbb{E}^{*}(\|X^*&(t_k)-Y^{0}(t_k)\|^2|\mathcal{F}^*_{t_{k-1}})= \|X^*(t_{k-1})-Y^0(t_{k-1})\|^2\\
	&+	\mathbb{E}^*(\|X^*(t_k)-X^*(t_{k-1})\|^2|\mathcal{F}^*_{t_{k-1}})+
	\mathbb{E}^{*}(\|Y^{0}(t_{k})-Y^{0}(t_{k-1})\|^2|\mathcal{F}^*_{t_{k-1}})\\
	&+ 2\mathbb{E}^*(\langle X^*(t_{k})-X^*(t_{k-1}),X^*(t_{k-1})-Y^0(t_{k-1})\rangle| \mathcal{F}^*_{t_{k-1}})\\ 
	&-2\mathbb{E}^{*}(\langle Y^0(t_{k})-Y^0(t_{k-1}),X^*(t_{k-1})-Y^0(t_{k-1})\rangle|\mathcal{F}^*_{t_{k-1}})\\ &-2\mathbb{E}^{*}(\langle X^*(t_k)-X^*(t_{k-1}),Y^0(t_k)-Y^0(t_{k-1})\rangle|\mathcal{F}^*_{t_{k-1}})\\ 
	\leq
	\|X^*&(t_{k-1})-Y^0(t_{k-1})\|^2\\
	&+ 2\mathbb{E}^*(\langle X^*(t_{k})-X^*(t_{k-1}),X^*(t_{k-1})-Y^0(t_{k-1})\rangle|\mathcal{F}^*_{t_{k-1}})\\ &-2\mathbb{E}^{*}(\langle Y^0(t_{k})-Y^0(t_{k-1}),X^*(t_{k-1})-Y^0(t_{k-1})\rangle|\mathcal{F}^*_{t_{k-1}})\\
	&+2\mathbb{E}^*(\|X^*(t_k)-X^*(t_{k-1})\|^2|\mathcal{F}^*_{t_{k-1}})+
	2\mathbb{E}^{*}(\|Y^{0}(t_{k})-Y^{0}(t_{k-1})\|^2|\mathcal{F}^*_{t_{k-1}}). 
	\end{split}
	\end{equation}
	Since (\ref{express:Y_0_martingale}) is a martingale for any $\phi\in\mathcal{D}$, using the definitions of the generator $\Lambda_t[u,v]$ and the function $g$ (see (\ref{intro:Lambda}) and (\ref{intro:g}) respectively), we get that
	\begin{equation*}
	\begin{split}\mathbb{E}^*(\langle Y^0(t_k)-Y^0(t_{k-1}),X^*&(t_{k-1})-Y^0(t_{k-1})|\mathcal{F}^*_{t_{k-1}})\rangle=\\ \mathbb{E}^*\Bigl(\int_{t_{k-1}}^{t_k}\int_{U\times V}\Lambda_t[u,v] &l_{Y^0(t_{k-1}),X^*(t_{k-1})}(Y^0(t))\eta_k^0(t,d(u,v))dt\Bigr|\mathcal{F}^*_{t_{k-1}}\Bigr)\end{split}
	\end{equation*} Here we denote $$l_{z_1,z_2}(x)\triangleq \langle x-z_1,z_2-z_1\rangle.$$ Since $$\Lambda_t[u,v]l_{z_1,z_2}(x)=\langle g(t,x,u,v),z_2-z_1\rangle,$$ we conclude that
	\begin{equation}\label{equal:Y_0_g}
	\begin{split}\mathbb{E}^*(\langle Y^0(t_k)-Y^0(t_{k-1}),X^*(t_{k-1})-Y^0&(t_{k-1})|\mathcal{F}^*_{t_{k-1}})\rangle=\\
	\mathbb{E}^*\Bigl(\int_{t_{k-1}}^{t_k}\int_{U\times V}\langle g(t,Y^0(t),u,v),X^*(&t_{k-1})-Y^0(t_{k-1})\rangle\eta_k^0(t,d(u,v))dt|\mathcal{F}^*_{t_{k-1}}\Bigr).
	\end{split}
	\end{equation}
	To simplify notation, put
	\begin{equation}\label{intro:u_v_k_minus}
	\hat{u}_k\triangleq u_\natural(s,X^*(t_{k-1}),Y^0(t_{k-1})),\ \ \hat{v}_k\triangleq v_\natural(s,X^*(t_{k-1}),Y^0(t_{k-1})).
	\end{equation} Note that $\hat{u}_k$ and $\hat{v}_k$ are random variable measurable w.r.t. $\mathcal{F}^*_{t_{k-1}}$.
	
	Combining  (\ref{ineq:x_t_x_s}), (\ref{ineq:Y_t_Y_s}), (\ref{equal:x_0_r_Y_r}) and (\ref{equal:Y_0_g}), we get 
	\begin{equation}\label{ineq:key_E_x_0_Y_r}
	\begin{split}
	\mathbb{E}^*(\|X^*&(t_{k})-Y^0(t_k)\|^2|\mathcal{F}^*_{t_{k-1}})\\
	\leq\|X^*&(t_{k-1})-Y^0(t_{k-1})\|^2+2(M)^2(r-s)^2+2\delta^2(r-s)\\
	+2&\tilde{\alpha}(r-s)\cdot(r-s)+
	2\int_s^r\langle f(t,X^*(t),\hat{u}_k,\hat{v}_k),X^*(t_{k-1})-Y^0(t_{k-1})\rangle dt\\- 2&\mathbb{E}^*\Bigl(\int_{t_{k-1}}^{t_k}\langle g(t,Y^0(t),u,v),X^*(t_{k-1})-Y^0(t_{k-1})\rangle\eta_k^0(t,d(u,v))dt|\mathcal{F}^*_{t_{k-1}}\Bigr).
	\end{split}\end{equation} 
	Further, from conditions (L4), (L6) and estimate (\ref{ineq:x_t_x_s}) we conclude that the following inequality is fulfilled $P^*$-a.s.:
	\begin{equation}\label{ineq:f_x_star_y_star}
	\begin{split}
	\langle f(t,X^*&(t),\hat{u}_k,\hat{v}_k),X^*(t_{k-1})-Y^0(t_{k-1})\rangle\\\leq \langle &f(t_{k-1},X^*(t_{k-1}),\hat{u}_k,\hat{v}_k),X^*(t_{k-1})-Y^0(t_{k-1})\rangle\\&+
	\alpha(t_k-t_{k-1})\|X^*(t_{k-1})-Y^0(t_{k-1})\|\\& +K\|X^*(t)-X^*(t_{k-1})\|\cdot\|X^*(t_{k-1})-Y^0(t_{k-1})\|\\ \leq \langle &f(t_{k-1},X^*(t_{k-1}),\hat{u}_k,\hat{v}_k),X^*(t_{k-1})-Y^0(t_{k-1})\rangle\\&+ \frac{1}{2}(\alpha(t_k-t_{k-1}))^2+\frac{(K)^2}{2}\|X^*(t)-X^*(t_{k-1})\|^2\\
	&+ \|X^*(t_{k-1})-Y^0(t_{k-1})\|^2\\ \leq \|&X^*(t_{k-1})-Y^0(t_{k-1})\|^2\\&+ \langle f(t_{k-1},X^*(t_{k-1}),\hat{u}_k,\hat{v}_k),X^*(t_{k-1})-Y^0(t_{k-1})\rangle\\&+ \frac{1}{2}(\alpha(t_k-t_{k-1}))^2+\frac{(KM)^2}{2}(t_k-t_{k-1})^2.
	\end{split}
	\end{equation}
	Analogously, condition (L4), (L6) and inequality (\ref{ineq:Y_t_Y_s}) imply the following estimate for all $u\in U$, $v\in V$,
	\begin{equation}\label{ineq:g_x_star_y_star}
	\begin{split}
	-\mathbb{E}^*\bigl(\langle g(t,&Y^0(t),u,v),X^*(t_{k-1})-Y^0(t_{k-1})\rangle\bigl|\mathcal{F}^*_{t_{k-1}}\bigr)\\ 
	\leq - &\mathbb{E}^*\bigl(\langle g(t_{k-1},Y^0(t_{k-1}),u,v),X^*(t_{k-1})-Y^0(t_{k-1})\rangle\bigl|\mathcal{F}^*_{t_{k-1}}\bigr)\\
	& +\alpha(t-s)\|X^*(t_{k-1})-Y^0(t_{k-1})\|\\ & 
	+K \|Y^0(t)-Y^0(t_{k-1})\|\cdot\|X^*(t_{k-1})-Y^0(t_{k-1})\|  \\ 
	\leq
	\|&X^*(t_{k-1})-Y^0(t_{k-1})\|^2
	\\&- \mathbb{E}^*\bigl(\langle g(t_{k-1},Y^0(t_{k-1}),u,v),X^*(t_{k-1})-Y^0(t_{k-1})\rangle\bigl|\mathcal{F}^*_{t_{k-1}}\bigr) \\ &+\frac{(K)^2}{2}\delta^2(t_k-t_{k-1})+\frac{1}{2}\alpha(t_k-t_{k-1}))^2\\&+ \frac{(K)^2}{2}\tilde{\alpha}(t_k-t_{k-1})\cdot(t_k-t_{k-1}).
	\end{split}
	\end{equation}
	Inequalities (\ref{ineq:key_E_x_0_Y_r}), (\ref{ineq:f_x_star_y_star}), (\ref{ineq:g_x_star_y_star}) yield the estimate
	\begin{equation}\label{ineq:E_x_0_Y_r_prefinal}
	\begin{split}
	\mathbb{E}^*(\|X^*(t_{k})-Y^0(t_k)\|^2|\mathcal{F}^*_{t_{k-1}})\hspace{100pt}&{}\\
	\leq
	\|X^*(t_{k-1})-Y^0(t_{k-1})\|^2(1+4(t_k-t_{k-1}))+ &2\delta^2(t_k-t_{k-1})\\+
	2\int_{t_{k-1}}^{t_k}\langle f(t_{k-1},X^*(t_{k-1}),\hat{u}_k,\hat{v}_k),X^*(&t_{k-1})-Y^0(t_{k-1})\rangle dt-\\
	- 2\mathbb{E}^*\Bigl(\int_{t_{k-1}}^{t_k}\int_{U\times V}\langle g(t_{k-1},Y^0(t_{k-1}),u,&v),\\X^*(t_{k-1})-Y^0(t_{k-1})\rangle&\eta^0_k(t,d(u,v))dt|\mathcal{F}^*_{t_{k-1}}\Bigr)\\+
	\epsilon(t_k-t_{k-1})\cdot (t_k-t_{k-1})\hspace{55pt}&{},\end{split}\end{equation}
	where the function $\epsilon$ is defined by (\ref{intro:epsilon}).
	
	From conditions \ref{cond:lip} and \ref{cond:delta} it follows that, for any $u\in U$, $v\in V$,
	\begin{equation*}
	\begin{split}
	\langle f(t_{k-1}&,X^*(t_{k-1}),\hat{u}_k,\hat{v}_k)- g(t_{k-1},Y^0(t_{k-1}),u,v),X^*(t_{k-1})-Y^0(t_{k-1})\rangle\\ 
	\leq \langle &f(t_{k-1},X^*(t_{k-1}),\hat{u}_k,\hat{v}_k)-g(t_{k-1},X^*(t_{k-1}),u,v),X^*(t_{k-1})-Y^0(t_{k-1})\rangle
	\\&+
	K\|X^*(t_{k-1})-Y^0(t_{k-1})\|^2\\ 
	\leq \langle &f(t_{k-1},X^*(t_{k-1}),\hat{u}_k,\hat{v}_k)-f(t_{k-1},X^*(t_{k-1}),u,v), X^*(t_{k-1})-Y^0(t_{k-1})\rangle\\&+
	\left(K+\frac{1}{2}\right)\|X^*(t_{k-1})-Y^0(t_{k-1})\|^2+\delta^2.
	\end{split}
	\end{equation*}
	Using (\ref{intro:u_nat_down}), (\ref{intro:v_nat_down}) and (\ref{intro:u_v_k_minus}), we get that, for all $u\in U$, $v\in V$,
	$$ \langle f(t_{k-1},X^*(t_{k-1}),\hat{u}_k,\hat{v}_k)-f(t_{k-1},X^*(t_{k-1}),u,v),X^*(t_{k-1})-Y^0(t_{k-1})\rangle\leq 0.$$ 
	This and (\ref{ineq:E_x_0_Y_r_prefinal}) imply
	\begin{multline*}
	\mathbb{E}^*(\|X^*(t_{k})-Y^0(t_k)\|^2|\mathcal{F}^*_{t_{k-1}})\leq
	\|X^*(t_{k-1})-Y^0(t_{k-1})\|^2(1+\beta(t_k-t_{k-1}))\\+4 \delta^2(t_k-t_{k-1})+
	\epsilon(t_k-t_{k-1})\cdot (t_k-t_{k-1})
	\end{multline*} where $\beta$ is defined by (\ref{intro:beta}).
	Hence, assumption (\ref{equal:Y_t_K_minus}) implies (\ref{ineq:E_star_cond}) for given $k=0,\ldots, n$. Thus, by construction of the profile of strategies $\mathfrak{w}^*$ under assumption (\ref{equal:Y_t_K_minus}) the equality $Y^0(t)=\mathcal{Y}^{*,1}(t)=\mathcal{Y}^{*,2}(t)$ holds true for $t\in [t_{k-1},t_k]$.
	
	To complete the proof it suffices to recall that 
	$Y^0(t_0)=\mathcal{Y}^{*,1}(t_0)=\mathcal{Y}^{*,2}(t_0)=X^*(t_0)=x_0$ and use the induction.
\end{proof}

Lemma \ref{lm:key_estimate} and the fact that $X^*(t_0)=Y^0(t_0)=x_0$ immediately imply the following.
\begin{corollary}\label{cor:X_star_T}
	$$E^*\|X^*(T)-Y^0(T)\|^2\leq 4\delta^2 Te^{\beta T}+\epsilon(d(\Delta))Te^{\beta T}. $$ 
\end{corollary}

Now let $\mathfrak{w}^1$ be an unilateral deviation from $\mathfrak{w}^*$ by the first player. 
This means (see Definition \ref{def:deviation}) that $\mathfrak{w}^1$ is a 6-tuple $(\Omega^1,\mathcal{F}^1,\{\mathcal{F}^1\}_{t\in [t_0,T]},u^1_{x(\cdot)},v^1_{x(\cdot)}, P^1_{x(\cdot)})$ such that, for some   filtered measurable space  $(\Omega',\mathcal{F}',\{\mathcal{F}'\}_{t\in [t_0,T]})$, the following properties hold true:
\begin{itemize}
	\item $\Omega^1=\Omega^*\times\Omega'$;
	\item $\mathcal{F}^1=\mathcal{F}^*\otimes\mathcal{F}'$;
	\item $\mathcal{F}^1_t=\mathcal{F}_t^*\otimes\mathcal{F}'_t$;
	\item for any $x(\cdot)\in C([t_0,T];\rd)$ and any $A\in\mathcal{F}^*$, $P^1_{x(\cdot)}(A\times \Omega')=P^*_{x(\cdot)}(A)$;
	\item for any $x(\cdot)\in C([t_0,T];\rd)$, $t\in [t_0,T]$, $\omega\in\Omega^*$, $\omega'\in\Omega'$, $v_{x(\cdot)}(t,\omega,\omega')=v_{x(\cdot)}^*(t,\omega)$.
\end{itemize}

Let $X^1$ and $P^1$ be generated by $\mathfrak{w}^1$ and $(t_0,x_0)$. Below $\mathbb{E}^1$ stands for the expectation according to $P^1$. Denote 
\begin{equation}\label{intro:Y_2}
\mathcal{Y}^{\natural,2}(t)\triangleq Y_{X^1(\cdot)}^2(t).
\end{equation} For $t\in [t_{k-1},t_k]$, put 
\begin{equation}\label{intro:S_2}
\mathcal{S}^2_k(t)\triangleq \overline{Y}^{2,t_{k-1},t_k}_{\mathcal{Y}^{\natural,2}(t_{k-1}),u^\natural(t_{k-1},X^1(t_{k-1}),\mathcal{Y}^{\natural,2}(t_{k-1}))}(t). 
\end{equation}

\begin{lemma}\label{lm:key_estimate_dev} The following inequality holds true, for $k=1,\ldots,n$,
	\begin{equation*}
	\begin{split}
	E^1(\|X^1(t_k)-\mathcal{S}^2_k(t_k)\|^2|\mathcal{F}^1_{t_{k-1}})\leq \|X^1(t_{k-1})-&\mathcal{Y}^{\natural,2}(t_{k-1})\|^2(1+\beta(t_k-t_{k-1}))\\&+ (4\delta^2+\epsilon(t_k-t_{k-1})) 	 \cdot(t_k-t_{k-1}).
	\end{split}
	\end{equation*}
\end{lemma}
\begin{proof}
	Notice (see (\ref{intro:v_star})) that, for $t\in [t_{k-1},t_k]$, the control of the second player is a random variable  equal to 
	\begin{equation}\label{intro:v_bar_k}
	\bar{v}_k\triangleq v_\natural(t_{k-1},X^1(t_{k-1}),\mathcal{Y}^{\natural,2}(t_{k-1})).
	\end{equation} By construction, $\bar{v}_k$ is a   measurable w.r.t. $\mathcal{F}^1_{t_{k-1}}$. Denote $\tilde{u}(t)\triangleq u^1_{X^1(\cdot)}(t)$. By Definition \ref{def:motion}, the following equality holds $P^1$-a.s.  
	\begin{equation}\label{express:X_1}
	\frac{d}{dt}X^1(t)=f_1(t,X^1(t),\tilde{u}(t))+f_2(t,X^1(t),\bar{v}_k).
	\end{equation}

	Denote 
	\begin{equation}\label{intro:u_bar_k}
	\bar{u}_k\triangleq u^\natural(t_{k-1},X^1(t_{k-1}),\mathcal{Y}^{\natural,2}(t_{k-1})). 
	\end{equation}
	
	Analogously, for $t\in [t_{k-1},t_k]$, put 
	\begin{equation}\label{intro:nu_k}
	\nu_k(t)\triangleq \nu^{t_{k-1},t_k}_{\mathcal{Y}^{\natural,2}(t_{k-1}),u^\natural(t_{k-1},X^1(t_{k-1}),\mathcal{Y}^{\natural,2}(t_{k-1}))}(t). 
	\end{equation}
	
	By Condition $(\mathcal{C})$ we have that, for any $\phi\in\mathcal{D}$,
	\begin{equation}\label{express:Y_dev_2}
	\phi(\mathcal{S}^2_k(t))-\int_{t_{k-1}}^t\int_V\Lambda_\tau[\bar{u}_k,v]\phi(\mathcal{S}^2_k(\tau)) \nu_k(\tau,dv)d\tau
	\end{equation} is $\{\mathcal{F}^1_t\}_{t\in [t_{k-1},t_k]}$-martingale.
	
	We have that
	\begin{equation}\label{ineq:X_1_t_k}
	\|X^1(t)-X^1(t_{k-1})\|\leq M(t-t_{k-1})\ \ P^1\text{-a.s}.
	\end{equation} Further, using conditions (L5), (L8), one can prove that
	\begin{equation}\label{ineq:Y_1_t_k}
	{\mathbb{E}}^1(\|\mathcal{S}^2_k(t)-\mathcal{S}^2_k(t_{k-1})\|^2|\mathcal{F}^1_{t_{k-1}})\leq \delta^2(t-t_{k-1})+\tilde{\alpha}(t-t_{k-1})\cdot (t-t_{k-1}).
	\end{equation} Here $\tilde{\alpha}(\cdot)$ is defined by (\ref{intro:alpha_3}). The proof of this statement is analogous to the proof of \cite[Lemma 13]{averboukh_SIAM}.
	
	In the same way as in the proof of Lemma \ref{lm:key_estimate},  inequalities (\ref{ineq:X_1_t_k}), (\ref{ineq:Y_1_t_k}), equality~(\ref{express:X_1}) and the facts that  (\ref{express:Y_dev_2}) is $\{\mathcal{F}^1_t\}_{t\in [t_{k-1},t_k]}$-martingales for $\phi(x)=\langle a,x\rangle$ and for $\phi(x)=\|x-a\|^2$ yield the estimate
	\begin{equation}\label{ineq:E_X_1_S_2_prefinal}
	\begin{split}
	\mathbb{E}^*(\|X^1(t_{k})-\mathcal{S}^2_k(t_k)\|^2|\mathcal{F}^1_{t_{k-1}})\hspace{40pt}&{}\\\leq
	\|X^1(t_{k-1})-\mathcal{Y}^{\natural,2}(t_{k-1})\|^2(1+4(&t_k-t_{k-1}))+ 2\delta^2(t_k-t_{k-1})\\+
	2\mathbb{E}^1\Bigl( \int_{t_{k-1}}^{t_k} \langle f(t_{k-1},X^1(t_{k-1}),\tilde{u}&(t),\bar{v}_k), X^*(t_{k-1})-\mathcal{Y}^{\natural,2}(t_{k-1})\rangle dt\Bigl|\mathcal{F}^1_{t_{k-1}}\Bigr)\\
	- 2\mathbb{E}^*\Bigl(\int_{t_{k-1}}^{t_k}\int_{V}\langle g(t_{k-1},\mathcal{Y}^{\natural,2}(t_{k-1})&,\bar{u}_k,v),\\
	&X^*(t_{k-1}) -\mathcal{Y}^{\natural,2}_k(t_{k-1})\rangle\nu_k(t,dv)dt\Bigl|\mathcal{F}^1_{t_{k-1}}\Bigr)\\+
	\epsilon(t_k-t_{k-1})\cdot (t_k-t_{k-1}).\hspace{34pt}&{}\end{split}\end{equation}
	In (\ref{ineq:E_X_1_S_2_prefinal}) we use the equality $\mathcal{Y}^{\natural,2}(t_{k-1})=\mathcal{S}^2_k(t_{k-1})$.
	
	Using  (\ref{intro:v_nat_down}), (\ref{intro:u_nat_up}), (\ref{intro:v_bar_k}), (\ref{intro:u_bar_k}) and  conditions \ref{cond:lip}, \ref{cond:delta}, we obtain the following inequality, for any $v\in V$:
	\begin{multline*}
	\langle f(t_{k-1},X^1(t_{k-1}),\tilde{u}(t),\bar{v}_k)- g(t_{k-1},\mathcal{Y}^{\natural,2}_k(t_{k-1}),\bar{u}_k,v),X^*(t_{k-1})-\mathcal{Y}^{\natural,2}_k(t_{k-1})\rangle\\\leq \left(K+\frac{1}{2}\right)\|X^*(t_{k-1})-Y^0(t_{k-1})\|^2+\delta^2.
	\end{multline*} This and (\ref{ineq:E_X_1_S_2_prefinal}) imply the statement of the lemma.
\end{proof}

Let $\Theta$ be a stopping time taking values in $\{t_0,\ldots,t_n\}$ defined by the rule:
$\Theta=t_{l-1},$ if \begin{equation}\label{ineq:decision_X_Y}\begin{split}E^1(\|X^1(t_j)-Y^0(t_j)\||\mathcal{F}^1_{t_{j-1}}&)\\\leq  \|X^1(t_{j-1})-Y^0&(t_{j-1})\|^2(1+\beta(t_{j}-t_{j-1}))\\&+(4\delta^2+\epsilon(t_{j}-t_{j-1}))
\cdot(t_{j}-t_{j-1})\end{split}\end{equation} is fulfilled for all $j=1,\ldots, l-1$ and violates for $j=l$. If (\ref{ineq:decision_X_Y}) is valid for all $j=1,\ldots,n$, then assume that $\Theta=t_n$.

From (\ref{intro:Psi}), (\ref{intro:Y_0}), (\ref{intro:Y_2}) it follows  that, given $k=1,\ldots,n$, and $t\in [t_{k-1},t_k]$, 
\begin{equation}\label{equal:Y_nat_2_Theta}
\mathcal{Y}^{\natural,2}(t)=\left\{\begin{array}{cc}
Y^0(t), & t_{k-1}<\Theta, \\ 
\mathcal{S}^2_k(t), & t_{k-1}\geq \Theta.
\end{array}\right.
\end{equation} This, (\ref{ineq:decision}), Lemma \ref{lm:key_estimate_dev} and equality $X^1(t_0)=\mathcal{Y}^{\natural,2}_k(t_0)=x_0$ $P^1$-a.s. give the following.

\begin{corollary}\label{cor:estimate_X_1_Y_2}
	For any $k=1,\ldots,n$,
	\begin{equation*}
	\begin{split}
	E^1(\|X^1(t_k)-\mathcal{Y}^{\natural,2}(t_k)\|^2|\mathcal{F}^*_{t_{k-1}})\leq \|X^1(t_{k-1}&)-\mathcal{Y}^{\natural,2}(t_{k-1})\|^2(1+\beta(t_k-t_{k-1}))\\&+(4\delta^2+\epsilon(t_k-t_{k-1})) 	 \cdot(t_k-t_{k-1}).
	\end{split}
	\end{equation*}	
	Moreover,
	$$E^1\|X^1(T)-\mathcal{Y}^{\natural,2}(T)\|\leq 4\delta^2 Te^{\beta T}+\epsilon(d(\Delta))Te^{\beta T}. $$
\end{corollary}

\section{Proof of the main result}\label{sect:proof_of_the_main_result}
In this section we prove that the strategy $\mathfrak{w}^*$ defined in Section \ref{sect:construction} is an approximate public-signal correlated  equilibrium.
\begin{proof}[Proof of Theorem \ref{th:near_Nash}]
	First, let us assume that the players use the profile of strategies~$\mathfrak{w}^*$. Let $X^*$, $P^*$ be generated by $\mathfrak{w}^*$ and initial position $(t_0,x_0)$. Recall that $Y^0$ is the stochastic process defined by (\ref{intro:Y_0}).

	For $t\in [t_{k-1},t_k)$, let $\eta^0(t)$ be equal to $\eta^0_k(t)$.
	Part (i) of Condition $(\mathcal{C})$ implies that, for $i=1,2$,
	$$c_i(t_k,Y^0(t_k))+\int_{t_0}^{t_k}\int_{U\times V}h_i(t,Y^0(t),u,v)\eta^0(t,d(u,v))dt $$ is a $\{\mathcal{F}^*_{t_k}\}_{k=0}^n$-martingale.
	Using this and the boundary condition, we get that
	\begin{equation}\label{equal:c_i_sigma_h}
	c_i(t_0,x_0)=\mathbb{E}^*\left[\gamma_i(Y^0(T))+\int_{t_0}^{T}\int_{U\times V}h_i(t,Y^0(t),u,v)\eta^0(t,d(u,v))dt\right]. 
	\end{equation}	
	The Lipschitz continuity of the functions $\gamma_i$, Corollary 2, and  Jensen's inequality imply 
	\begin{multline}\label{ineq:sigma_X_star_prefinal}
	\mathbb{E}^*|\gamma_i(X^*(T))-\gamma_i(Y^0(T))|\leq R\mathbb{E}^*\|X^*(T)-Y^0(T)\|\\\leq R\sqrt{\mathbb{E}^*\|X^*(T)-Y^0(T)\|^2}\leq R\sqrt{C^2\delta^2+\epsilon(d(\Delta))Te^{\beta T}}.
	\end{multline}
	Here the constant $C$ is defined by (\ref{intro:C}).
	
	Since $|h_i(t,x,u,v)|\leq \delta$, using (\ref{equal:c_i_sigma_h}), (\ref{ineq:sigma_X_star_prefinal}) and Jensen's inequality, we get
	\begin{multline}\label{estimate:sigma_X_c_i}
	|\mathbb{E}^*\gamma_i(X^*(T))-c_i(t_0,x_0)|\\\leq \mathbb{E}^*\left|\gamma_i(X^*(T))- \gamma_i(Y^0(T))-\int_{t_0}^{T}\int_{U\times V}h_i(t,Y^0(t),u,v)\eta^0(t,d(u,v))dt\right|\\ \leq R\sqrt{C^2\delta^2+\epsilon(d(\Delta))Te^{\beta T}}+T\delta.
	\end{multline}

	Now assume that the first player deviates. Let $\mathfrak{w}^1$ be an unilateral deviation of the first player, $X^1$, $P^1$ be generated by $\mathfrak{w}^1$ and initial position $(t_0,x_0)$. Recall that $\mathcal{Y}^{\natural,2}$ and $\mathcal{S}^2_k$ are defined by~(\ref{intro:Y_2}),~(\ref{intro:S_2}) respectively. Let us introduce the generalized control $\eta^2_k$ on $[t_{k-1},t_k)$ by the rule:
	\begin{equation}\label{intro:eta_2}
	\eta^2_k(t,d(u,v))\triangleq \left\{\begin{array}{cc}
	\eta^0_k(t,d(u,v)), & t_{k-1}<\Theta, \\ 
	\delta_{\bar{u}_k}\otimes \nu_k(t,dv), & t_{k-1}\geq \Theta.
	\end{array}\right. 
	\end{equation}
	Here $\eta^0_k$, $\bar{u}_k$ and $\nu_k$ are defined by (\ref{intro:eta_0}), (\ref{intro:u_bar_k}) and (\ref{intro:nu_k}) respectively. Recall that the stopping time $\Theta$ is equal to $t_{l-1}$ when (\ref{ineq:decision_X_Y}) holds true for all $j=0,\ldots,l-1$ and violates for $j=l$.
	Finally, define the control $\eta^2$ as follows: if $t\in [t_{k-1},t_{k}]$, then
	$$\eta^2(t)\triangleq \eta^2_k(t). $$
	
	We claim that
	\begin{equation}\label{express:c_1_deviation}
	c_1(t_k,\mathcal{Y}^{\natural,2}(t_k))+\int_{t_0}^{t_{k}}\int_{U\times V}h_1(t,\mathcal{Y}^{\natural,2}(t),u,v)\eta^2(t,d(u,v))dt
	\end{equation} is a $\{\mathcal{F}^1_{t_k}\}_{k=0}^n$-supermartingale w.r.t. $P^1$.
	Indeed, it suffices to prove that
	\begin{multline}\label{ineq:c_1_martingale}
	\mathbb{E}^1\Bigl[\Bigl(c_1(t_k,\mathcal{Y}^{\natural,2}(t_k))+ \int^{t_{k}}_{t_{k-1}}\int_{U\times V}h_1(t,\mathcal{Y}^{\natural,2}(t),u,v)\eta^2(t,d(u,v))dt\Bigr)\Bigr| \mathcal{F}^1_{t_{k-1}}\Bigr]\\\leq c_1(t_{k-1},\mathcal{Y}^{\natural,2}(t_{k-1})).
	\end{multline}
	Using (\ref{equal:Y_nat_2_Theta}) and (\ref{intro:eta_2}), we conclude that
	\begin{multline*}
	\mathbb{E}^1\Bigl[\Bigl(c_1(t_k,\mathcal{Y}^{\natural,2}(t_k))+ \int\limits_{t_{k-1}}^{t_{k}}\int_{U\times V}h_1(t,\mathcal{Y}^{\natural,2}(t),u,v)\eta^2(t,d(u,v))dt \Bigr) \Bigr|\mathcal{F}^1_{t_{k-1}}\Bigr]\\=
	\mathbb{E}^1\Bigl[\Bigl(c_1(t_k,\mathcal{Y}^{\natural,2}(t_k))+ \int_{t_{k-1}}^{t_{k}}\int_{U\times V}h_1(t,\mathcal{Y}^{\natural,2}(t),u,v)\eta^2(t,d(u,v))dt \Bigr) \mathbf{1}_{t_{k-1}<\Theta}\Bigr|\mathcal{F}^1_{t_{k-1}}\Bigr]\\
	+\mathbb{E}^1\Bigl[\Bigl(c_1(t_k,\mathcal{Y}^{\natural,2}(t_k))+ \int_{t_{k-1}}^{t_{k}}\int_{U\times V}h_1(t,\mathcal{Y}^{\natural,2}(t),u,v)\eta^2(t,d(u,v))dt\Bigr) \mathbf{1}_{t_{k-1}\geq\Theta}\Bigr|\mathcal{F}^1_{t_{k-1}}\Bigr]\\=
	\mathbb{E}^1\Bigl[\Bigl(c_1(t_k,Y^0(t_k))+ \int_{t_{k-1}}^{t_{k}}\int_{U\times V}h_1(t,Y^0(t),u,v)\eta^0_k(t,d(u,v))dt\Bigr) \mathbf{1}_{t_{k-1}<\Theta}\bigr|\mathcal{F}^1_{t_{k-1}}\Bigr]\\
	+\mathbb{E}^1\Bigl[\Bigl(c_1(t_k,\mathcal{S}^2_k(t_k))+ \int_{t_{k-1}}^{t_{k}}\int_{ V}h_1(t,\mathcal{S}^{2}_k(t),\bar{u}_k,v)\nu_k(t,dv)dt\Bigr) \mathbf{1}_{t_{k-1}\geq\Theta}\Bigr|\mathcal{F}^1_{t_{k-1}}\Bigr].
	\end{multline*}
	Part (i) of Condition $(\mathcal{C})$ and definitions of $Y^0$ (see (\ref{intro:Y_0})) and $\eta^0_k$ (see (\ref{intro:eta_0})) yield that
	\begin{multline*}
	\mathbb{E}^1\Bigl[\Bigl(c_1(t_k,Y^0(t_k))+ \int_{t_{k-1}}^{t_{k}}\int_{U\times V}h_1(t,Y^0(t),u,v)\eta^0_k(t,d(u,v))dt\Bigr) \mathbf{1}_{t_{k-1}<\Theta}\Bigr|\mathcal{F}^1_{t_{k-1}}\Bigr]\\ = c_1(t_{k-1},Y^0(t_{k-1}))\mathbf{1}_{t_{k-1}<\Theta}=
	c_1(t_{k-1},\mathcal{Y}^{\natural,2}(t_{k-1})) \mathbf{1}_{t_{k-1}<\Theta}.
	\end{multline*}
	Further, using part (iii) of Condition $(\mathcal{C})$, definitions of $\mathcal{S}_k^2$ (see (\ref{intro:S_2})) and $\nu_k$ (see (\ref{intro:nu_k})), we get
	\begin{multline*}
	\mathbb{E}^1\Bigl[\Bigl(c_1(t_k,\mathcal{S}^2_k(t_k))+ \int_{t_{k-1}}^{t_{k}}\int_V h_1(t,\mathcal{S}^{2}_k(t),\bar{u}_k,v)\nu_k(t,dv)dt\Bigr) \mathbf{1}_{t_{k-1}\geq\Theta}\Bigr|\mathcal{F}^1_{t_{k-1}}\Bigr]\\ \leq c_1(t_{k-1},\mathcal{Y}^{\natural,2}(t_{k-1})) \mathbf{1}_{t_{k-1}\geq \Theta}.
	\end{multline*}
	Thus, (\ref{ineq:c_1_martingale}) is fulfilled. It implies that the discrete time process (\ref{express:c_1_deviation}) is a $\{\mathcal{F}^1_{t_k}\}_{k=0}^n$-supermartingale.
	Therefore, using boundary condition, we get
	\begin{equation}\label{ineq:c_1_deviation}
	\mathbb{E}^1\left(\gamma_1(\mathcal{Y}^{\natural,2}(T))+ \int_{t_0}^{T}\int_{U\times V}h_1(t,\mathcal{Y}^{\natural,2}(t),u,v)\eta^2(t,d(u,v))dt\right)\leq c_1(t_0,x_0).
	\end{equation}
	By Jensen's inequality, the Lipschitz continuity of the function $\gamma_1$ and Corollary \ref{cor:estimate_X_1_Y_2} we have
	\begin{multline*}
	\mathbb{E}^1\gamma_1(X^1(T))-\mathbb{E}^1\gamma_1(\mathcal{Y}^{\natural,2}(T))\leq \mathbb{E}^1|\gamma_1(X^1(T))-\gamma_1(\mathcal{Y}^{\natural,2}(T))|\\\leq R\mathbb{E}^1\|X^1(T)-\mathcal{Y}^{\natural,2}(T)\|\leq R\sqrt{C^2\delta^2+\epsilon(d(\Delta))Te^{\beta T}}.
	\end{multline*} This, the fact that (\ref{express:c_1_deviation}) is a supermartingale and the estimate $|h_1(t,x,u,v)|\leq\delta$ (see (L8)) imply that
	\begin{equation}\label{ineq:sigma_1_c_1}
	\mathbb{E}^1\gamma_1(X^1(T))\leq c_1(t_0,x_0)+R\sqrt{C^2\delta^2+\epsilon(d(\Delta))Te^{\beta T}}+T\delta.
	\end{equation}
	
	In the same way one can consider the case when the second player deviates. If $\mathfrak{w}^2=(\Omega^2,\mathcal{F}^2,\{\mathcal{F}^2\}_{t\in [t_0,T]},u^2_{x(\cdot)},v^2_{x(\cdot)}, P^2_{x(\cdot)})$ is an unilateral deviation of the second player, $X^2$ and $P^2$ are generated by $\mathfrak{w}^2$ and initial position $(t_0,x_0)$, then \begin{equation}\label{ineq:sigma_2_c_2}
	\mathbb{E}^2\gamma_2(X^2(T))\leq c_2(t_0,x_0)+R\sqrt{C^2\delta^2+\epsilon(d(\Delta))Te^{\beta T}}+T\delta.
	\end{equation} The proof of this property relies on  analogs of Lemma \ref{lm:key_estimate_dev} and Corollary \ref{cor:estimate_X_1_Y_2}.

	Inequalities (\ref{estimate:sigma_X_c_i}), (\ref{ineq:sigma_1_c_1}) and (\ref{ineq:sigma_2_c_2}) imply that $\mathfrak{w}^*$ is a public-signal correlated $\varepsilon$-equilibrium for any $\varepsilon$ such that $$\varepsilon\geq R\sqrt{C^2\delta^2+\epsilon(d(\Delta))Te^{\beta T}}+T\delta. $$ Since we can construct the profile of strategies $\mathfrak{w}^*$ using a partition $\Delta$ with an arbitrary small fineness, the statement of Theorem \ref{th:near_Nash} holds true.
\end{proof}

\section*{Acknowledgments}
I am thankful to the anonymous referees for their valuable and profound comments.

\bibliography{Diffgames_Nash}
\end{document}